\documentclass[11pt]{article}
\usepackage[usenames,dvipsnames]{color}
\usepackage[normalem]{ulem}
\usepackage{bm}
\usepackage{fullpage,amsfonts,amsmath,amsthm,amssymb,mathtools,mathrsfs,graphicx,upgreek}

\usepackage{hyperref}
\usepackage{cleveref}
\crefname{subsection}{subsection}{subsections}

\usepackage[left=2.54cm,top=2.54cm,right=2.54cm,bottom=2.54cm]{geometry}

\newcommand{\CM}[1]{{\color{black}{#1}}}
\newcommand{\scr}{\mathcal}
\newcommand{\mb}{\mathbb}
\newcommand{\til}{\widetilde}

\newcommand{\whp}{{\it w.h.p.}}

\newcommand{\st}{{s.t.}}

\newcommand{\lrpar}[1]{\left( #1 \right)}

\newcommand{\eps}{\epsilon}
\newcommand{\bP}{\mathbb{P}}

\newcommand{\bE}{\mathbb{E}}

\newcommand{\cS}{\mathcal{S}}

\newcommand{\cP}{\mathcal{P}}
\newcommand{\cM}{\mathcal{M}}

\newcommand{\cH}{\mathcal{H}}

\newcommand{\Sc}{\cS^c}

\renewcommand{\Pr}{\mb{P}}

\renewcommand{\st}{m^*}
\newcommand{\m}{m}
\newcommand{\ms}{m^*}

\newcommand{\ES}[1]{{\color{black}{#1}}}

\newcommand{\tsn}{T_{\cS_n}}
\newcommand{\bin}{\mathrm{Bin}}

\newcommand{\supp}{\mathrm{Supp}}

\newtheorem{thm}{Theorem}
\newtheorem{lemma}[thm]{Lemma}
\newtheorem{cor}[thm]{Corollary}

\newtheorem{question}{Question}
\newtheorem{proposition}[thm]{Proposition}

\usepackage{algorithm}
\usepackage{algpseudocode}

\newenvironment{varalgorithm}[1]
  {\algorithm\renewcommand{\thealgorithm}{#1}}
  {\endalgorithm}

\newcommand{\D}{\scr{D}}

\usepackage{enumitem}
\newcommand{\subscript}[2]{$#1 _ #2$}

\theoremstyle{definition}
\newtheorem{definition}{Definition}
\theoremstyle{definition}
\newtheorem{remark}[thm]{Remark}

\title{Sharp Thresholds in Adaptive Random Graph Processes}

\author{Calum MacRury
\thanks{Graduate School of Business, Columbia University, New York, NY, USA,
\texttt{cm4379@columbia.edu}.  The majority of the author's work done on this paper was while they were affiliated with the Department of Computer Science at the University of Toronto.}
\and
Erlang Surya
\thanks{Department of Mathematics, University of California, San Diego, CA, USA,
\texttt{esurya@ucsd.edu}. Supported by NSF grant DMS-2225631.}
}

\date{}

\begin{document}
	
\maketitle



\begin{abstract}
The $\mathcal{D}$-process is a single player game in which the player is initially presented the empty graph on $n$ vertices. In each step, a subset of edges $X$ is independently sampled according to a distribution $\mathcal{D}$. The player then selects one edge $e$ from $X$, and adds $e$ to its current graph. For a fixed monotone increasing graph property $\scr{P}$, the objective of the player is to force the graph to satisfy $\scr{P}$ in as few steps as possible. The $\mathcal{D}$-process generalizes both the Achlioptas process and the semi-random graph process.

We prove a sufficient condition for the existence of a sharp threshold for $\mathcal{P}$ in the $\mathcal{D}$-process. 
Using this condition, in the semi-random process we prove the existence of a sharp threshold when $\mathcal{P}$ corresponds to being Hamiltonian or to containing a perfect matching. This resolves two of the open questions proposed by Ben-Eliezer et al. (RSA, 2020).
\end{abstract}


\newpage

\section{Introduction}

Let $n \in \mb{N}$, and $K_n$ be the complete graph on vertex set $[n] :=\{1,\ldots ,n\}$. Suppose that $\D$ is a fixed distribution on (non-empty) subsets of edges of $K_n$.
The $\D$-\textbf{adaptive random graph process} (shortly, $\D$-process) is a single player game in which the player is initially presented a graph $G_0$ on vertex set $[n]$, which unless specified otherwise, will be the empty graph. In each \textbf{step} (or round) $t \in \mb{N}$, a subset of edges $X_t$ is sampled from $\D$. The player (who is aware of graph $G_t$ and the subset $X_t$) must then select an edge $Y_t$ from $X_t$ and add it to $G_{t-1}$ to form $G_{t}$.
In this paper, the goal of the player is to devise a strategy which builds a (multi)graph satisfying a given monotone increasing property $\scr{P}$ in as few rounds as possible. Some examples of $\D$-processes are the Erdős–Rényi random graph process \cite{erdHos1960evolution} (where multi-edges are allowed), the Achlioptas process \cite{bohman2001avoiding}, the semi-random graph process \cite{ben2020semi}  (see Section \ref{sec:application_semi_random}), and the semi-random tree process \cite{burova2022semi}. 

Formally, a \textbf{strategy} (i.e., algorithm) $\scr{S} = \scr{S}_n$ is defined by specifying a sequence of functions $(s_{t})_{t=1}^{\infty}$, where for each $t \in \mb{N}$, $s_t(G_{t-1},X_t)$ is a distribution on $X_t$ which depends on the graph at step $t-1$ (and the edges of $X_t$). Then, an edge $Y_t \in X_t$ is chosen according to this distribution. If $s_t$ is an atomic distribution, then $Y_t$ is determined by $G_{t-1}$ and $X_t$. Note that if for each $t \ge 1$, $s_t$ is atomic, then
we say that the strategy $\scr{S}$ is \textbf{deterministic}. In this case, we may assume that each $s_t$ is a function which depends only on $X_1, \ldots , X_t$.

We denote $(G_{i}^{\scr{S}}(n))_{i=0}^{t}$ as the sequence of random (multi)graphs obtained by following the strategy $\scr{S}$ for $t$ rounds; where we shorten $G_{t}^{\scr{S}}(n)$
to $G_t$ or $G_{t}(n)$ when there is no ambiguity. Moreover, we define the stopping time $T_{\scr{S}} = T_{\scr{S}}(n)$ to be the minimum $t \ge 1$
such that $G_{t}^{\scr{S}}(n)$ satisfies $\scr{P}$, where $T_{\scr{S}} := \infty$ if no such $t$ exists. 
All our asymptotics are with respect to $n \rightarrow \infty$, and \textbf{with high probability} (w.h.p.) means with probability tending to $1$
as $n \rightarrow \infty$.

\begin{definition}[Sharp Threshold]	
	Given an edge monotonic property $\scr{P}$, we say that
	there exists a \textbf{sharp threshold} for $\scr{P}$ in the $\D$-process (or $\scr{P}$
	admits a sharp threshold),
	provided there exists a function $\st = \st_{\cP,\D}(n)$ such that for every $\eps>0$:
	\begin{enumerate}
	    \item There \textit{exists} a strategy $\cS'_n$ such that $\Pr(T_{\cS'_n}\le (1+\eps)\st)=1-o(1)$.
	    \item \textit{Every} strategy $\cS_n$ satisfies $\Pr(\tsn\le (1-\eps)\st)=o(1)$. 
	\end{enumerate}
When $\st$ satisfies these conditions, we say that it is a sharp threshold of $\cP$ in the $\D$-process.
\end{definition}

In general, it is difficult to show that a property has a sharp threshold for a $\D$-process due to the fact that one needs to prove a ``negative" result that \textit{any} algorithm which executes for too few rounds has win probability at most $o(1)$. This contrasts with many sharp threshold results
for (non-adaptive) random graph processes, where the negative result follows easily via the first-moment method. In this paper, we develop a framework which allows us to show that if a property satisfies a ``positive'' result involving the existence of a certain algorithm, then it must admit a sharp threshold. We formalize this framework in a condition we refer to as \textit{edge-replaceability} (see \Cref{def:edge_replaceable}).

There have been a few results which establish the existence of sharp thresholds for adaptive random graph processes. 
In \cite{ben-eliezer_fast_2020,ben-eliezer_fast_journal_2020}, Ben-Eliezer et al. showed that the property of containing an arbitrary spanning graph with $(1+o(1))\Delta n/2$ edges has $\Delta n/2$ as a sharp threshold, provided its maximum degree $\Delta$ satisfies $\Delta=\omega(\log n)$. For certain types of Achlioptas processes, Krivelevich et al. \cite{krivelevich2010hamiltonicity} \footnote{The model in \cite{krivelevich2010hamiltonicity} samples
$k$ edges uniformly at random from the set of currently missing edges instead of the set of all edges (i.e. $\D$ slightly changes over time). However, this distinction does not change the existence of sharp thresholds in the regime of interest.} showed that the property of being Hamiltonian admits a sharp threshold. Both of these papers follow the same high-level approach:
\begin{enumerate}
	\item A naive lower bound $L$ is obtained from a standard analysis, such that any strategy needs at least approximately $L$ steps to succeed w.h.p..
	\item An explicit strategy $\cS$ is devised and shown to satisfy the desired property in approximately $L$ steps w.h.p. This second
	step establishes the existence of a sharp threshold, and essentially all the work is done here.
\end{enumerate}
As one might expect, there are certain limitations to such an explicit approach. First, it is not always the case that the naive lower bound is the right answer. For instance, in the semi-random graph process, Gao et al. \cite{gao2020hamilton,gao2021perfect, gao_fully_2022} established that the naive lower bound can be improved substantially when the property  corresponds to containing a Hamiltonian cycle \cite{gao2020hamilton, gao_fully_2022},
or to containing a perfect matching \cite{gao2021perfect}. Second, even if a sharp threshold does exist, it is not clear that it can be identified by a strategy with an explicit description. For example, in the semi-random graph process, Gao et al. introduce algorithms for constructing Hamiltonian cycles \cite{gao2020hamilton,gao_fully_2022}, and perfect matchings \cite{gao2021perfect}.  While each algorithm
satisfies the relevant property in number of steps close to the best known lower bound, the authors indicate that they do \textit{not} believe their algorithms are optimal. This greatly limits their usefulness in terms of proving the existence of sharp thresholds. In this paper, we circumvent these limitations by developing a general machinery which allows us to establish the existence of sharp thresholds in the $\D$-process \textit{implicitly}. That is, without explicitly identifying lower bounds or finding (asymptotically) optimal strategies. 


While we are unaware of any work applying this implicit approach to any ``truly adaptive'' random graph process,
it has been used when the player has no real control (i.e., $\scr{D}$ is supported on singletons). In his seminal paper, Friedgut \cite{friedgut1999sharp} proved the existence of sharp thresholds for ``global" properties in the Erdős–Rényi random graph process in an implicit way (the model he considered does not allow multi-edges, but this distinction is irrelevant in many regimes of interest). To do so, he identifies the Erdős–Rényi random graph with the product space measure on $\{0,1\}^{\binom{n}{2}}$, and applies Fourier analysis to the Boolean function indicating
whether or not the random graph satisfies the given property. 
It is not clear how such techniques can be generalized to the $\D$-process, as in general, the $\D$-process depends on the decisions of
the player, and so it cannot be obviously modelled by a product space measure.



\subsection{Main result}
Given an arbitrary distribution $\D$, Theorem \ref{stealing} provides a sufficient condition for when a monotone increasing property $\cP$ admits a sharp threshold in the $\D$-process. For any $\theta\in (0,1)$, define $\m_{\scr{P}}(\theta,n)$ to be the minimum $t \ge 1$,
such that there exists a strategy $\scr{S}'_n$ which 
satisfies $ \bP[T_{\scr{S}'_n}\le \m_{\cP}(\theta,n) ]\ge \theta$, and for every strategy $\scr{S}_n$,
$\bP[ T_{\scr{S}_n} \le \m_{\cP}(\theta,n) -1 ] < \theta.$
For convenience we let $\ms:=\ms(n):=\m_{\cP}(1/2,n)$, and define the sufficient condition used in Theorem \ref{stealing}:
\begin{definition}[Edge-Replaceable] \label{def:edge_replaceable}
We say that $\cP$ is \textbf{$\omega$-edge-replaceable} (or just \textbf{edge-replaceable}) 
if there exists ${\omega:=\omega(n)\to\infty}$ such that the following guarantee holds:
 For any $G\in \cP$, and $e\in E(G)$, if we begin the $\D$-process with graph $G_0=G-e$, then there exists a strategy for the player which constructs some $G' \in \scr{P}$ in $\sqrt{\ms}/\omega$ steps with probability at least $1 - o(1/\sqrt{\ms})$. We refer to this strategy as an \textbf{edge-replacement procedure} of $\scr{P}$. 

\end{definition}
We remark that for the properties we consider, $G'$ will typically be distinct from $G$.


\begin{thm}[Sharp threshold] \label{stealing}
	If $\cP$ is $\omega$-edge-replaceable, then for any constants $0<\theta_1<\theta_2<1$, we have that
	\begin{equation}\label{eq:sharptransition}
	\m_{\scr{P}}(\theta_2,n)-\m_{\scr{P}}(\theta_1,n)=O_{\theta_1,\theta_2}\left(\frac{\ms}{\omega}\right),
	\end{equation}
	where the implicit constant in the $O$ term depends on $\theta_1,\theta_2$. Thus, $\ms$ is a sharp threshold of $\cP$.
\end{thm}
We prove Theorem \ref{stealing} by fixing an arbitrary strategy which succeeds with probability at least $\theta\in [\theta_1,\theta_2]$ in $\m_{\cP}(\theta,n)$ steps,
and identifying a strategy modification which has the \textit{potential} to increase the strategy's winning probability by $\Omega(1/\sqrt{\ms})$.
The proof relies on a martingale concentration inequality, whose full statement we defer to Section \ref{sec:approx_balanced}.
After performing this strategy modification, the final graph $G_0$ we are left with may be lacking an edge $e$ necessary to satisfy $\scr{P}$, 
however we can apply an augmentation via the edge-replacement procedure of $\scr{P}$ to $G_0$ to recover
a graph $G'$ which \textit{does} satisfy $\scr{P}$ in $\sqrt{\ms}/\omega$ steps with probability at least $1 - o(1/\sqrt{\ms})$. Thus,
we boost the win probability of $\scr{S}$ by $\Theta(1/\sqrt{\ms})$ in $\sqrt{\ms}/\omega$ steps.
By applying this procedure $\sqrt{\ms}$ times, we increase the original strategy's win probability by $\Theta(1)$. Since this only requires an extra $\Theta(\sqrt{\ms})\cdot \sqrt{\ms}/\omega=o(\ms)$ steps in total, we are able to establish the existence of a sharp threshold.

\subsection{Application: The Semi-Random Graph Process} \label{sec:application_semi_random}

The \textbf{semi-random graph process} was suggested by Peleg Michaeli, introduced formally in~\cite{ben2020semi}, and studied in~\cite{ben-eliezer_fast_2020,gao2020hamilton,ben-eliezer_fast_journal_2020,gao2021perfect,behague2021subgraph, gao_fully_2022,burova2022semi}. The process is a one player game in which
the player begins with the empty graph on $[n]$. In each step $t \ge 1$, the player is given a vertex $u_t$ drawn independently and uniformly at random (u.a.r.) from $[n]$, 
often referred to as a \textbf{square}. They then adaptively pick a vertex $v_t$ (called a \textbf{circle}), and add the edge $(u_t,v_t)$ to their current graph. 
Observe that if $\D$ is the uniform distribution over all spanning stars on $K_n$, then the $\D$-process encodes the semi-random graph process. 

\CM{To warm-up, we first consider the property $\scr{P}_k$ of attaining minimum degree $k \ge 1$ in the semi-random
graph process. In \cite{ben2020semi}, Ben-Eliezer et al. identified an explicit constant $h_k$ such that $h_k n$ is a sharp threshold for $\scr{P}_k$. Clearly, $\scr{P}_k$ is $\omega$-edge-replaceable with $\omega = \sqrt{n}/2$. Thus, applying \Cref{stealing} to $\scr{P}_k$ yields an alternative proof of the existence of a sharp threshold.}

Moving to our main applications, let $\cM$ be the property of containing a perfect matching (by perfect matching on an odd number of vertices, we mean a matching which saturates all but one vertex), and $\cH$ be the property of containing a Hamiltonian cycle. As an application of Theorem \ref{stealing} and the tools we develop in Section \ref{sec:proof_semi_random}, we prove the following sharp threshold result. This result answers two of the open problems proposed by Ben-Eliezer et al. in \cite{ben-eliezer_fast_journal_2020}
(the journal version of \cite{ben-eliezer_fast_2020}).
\begin{thm}\label{mainsemirandom}
	Let $\cP\in \{\cM, \cH\}$. In this case, if $\ms(n) :=\ms_{\cP}(n):= m_{\scr{P}}(1/2,n)$, then
	\begin{enumerate}
	    \item `Existence of a threshold': In the semi-random graph process, $\ms$ is a sharp threshold for $\scr{P}$ \label{thm:sharp_thresold_existence}
	    \item `Linear growth': There exists some constant $C_\cP > 0$, such that $\ms = (C_\cP  +o(1)) n$. \label{thm:sharp_threshold_linear}
	\end{enumerate}
\end{thm}

There are a few notable complications in proving Theorem \ref{mainsemirandom}. First, it turns out that the condition in Theorem \ref{stealing} does \textit{not} hold for either $\cM$ or $\cH$. However, for each of $\cM$ and $\cH$, we can define an \textit{approximate} property that does satisfy the required conditions, and thus admits
a sharp threshold. Since each approximate property is closely related to $\cM$ and $\cH$, we are able to argue that $\cM$ and $\cH$ also have sharp thresholds. Relating each approximate property with its ``full'' property relies on the ``clean-up'' algorithms of Gao et al. \cite{gao2021perfect,gao_fully_2022}. When $\cP$ is $\cM$, this clean-up algorithm allows one to extend a large matching to a perfect matching in a sublinear number of steps. When $\cP$ is $\cH$, the clean-up algorithm has a similar guarantee.

We emphasize that it is easy to establish $m^*=\Theta(n)$ for both properties $\cM$ and $\cH$. Hence the first part of Theorem~\ref{mainsemirandom} implies the existence of a function $C_{\cP}(n)=\Theta(1)$ such that $m^*=(C_{\cP}(n)+o(1))n$. The second part of the theorem shows that $\lim_{n\to\infty} C_{\cP}(n)$ exists. Showing the existence of such limit is non-trivial. Indeed, showing the limit of such function is the central question in many topics \cite{bayati2010combinatorial, ding2015proof}. We achieved this by considering the optimal strategy $\cS_n$ that minimizes $\bE \tsn=:I_n$, then showing that $I_n$ satisfies a certain set of inequalities (see \Cref{lem:fekete_semi}). We then use a purely analytic argument to show the existence of the limit $\lim_{n\to\infty} I_n/n$ (\Cref{lem:fekete_technical}), which quickly leads to the desired result. This is reminiscent of the approach seen in \cite{bayati2010combinatorial}, where a more standard subadditive inequality is combined with the brilliant use of interpolation to prove the existence of scaling limit for the size of independent sets in sparse random graphs. In our setting,
it is much easier to show that the required inequalities hold by standard ``strategy stealing'' arguments.

We conclude with the observation that a sharp threshold does not always exist for the semi-random process, which follows as a byproduct of the result in \cite{behague2021subgraph}.
\begin{thm}\label{coarsesimplified}
		Let $L$ be a fixed (finite) list of fixed graphs, none of which are forests.  Then in the semi-random graph process the property of containing a graph in $L$ does not admit a sharp threshold. 
\end{thm}

The rest of the paper is organized as follows. In Section 2 we prove Theorem~\ref{stealing}. In Section~3 we establish the martingale concentration inequality needed to prove Theorem~\ref{stealing}. In Section~4 we prove Theorems~\ref{mainsemirandom} and~\ref{coarsesimplified}. Section~5 list some final remarks.
	

\section{Proving Theorem \ref{stealing}} \label{sec:main}
	

Suppose that $\scr{P}$ is an edge-replaceable property with respect to $\omega \rightarrow \infty$ (see Definition \ref{def:edge_replaceable})
in some arbitrary $\scr{D}$-process. Moreover, take $0<\theta_1<\theta_2<1$. We wish to show
that if we are given a strategy which wins after $\m(\theta_1,n)$ steps with probability at least $\theta_1$,
then we can augment the strategy to boost its win probability to $\theta_2$ in $O\left(\ms/\omega\right)$ additional steps.
If we can prove this, then it will imply that  $\m \left( \theta_2,n \right)-\m(\theta_1,n) = O\left(\ms /\omega\right)$.
Now, suppose that we have boosted to a win probability of $\theta \in [\theta_1, \theta_2]$, and 
$\theta^* :=  \min(\theta+\theta(1-\theta)^3/32,\theta_2)$ is the next \textit{target} probability we wish to
boost to. We claim that this increase is attainable in an appropriate number of steps. That is,
for each $\theta \in [\theta_1, \theta_2]$,
\begin{equation} \label{eqn:large_explosion_overview}
\m \left( \theta^*,n \right)-\m(\theta,n) \le \ms/\omega.
\end{equation}
By beginning with $\theta = \theta_1$,
and iterating \eqref{eqn:large_explosion_overview} a constant number of times, \Cref{stealing} follows
(see the proof of \Cref{stealing} below for the details). 
\subsubsection{Reducing \eqref{eqn:large_explosion_overview} to Small Boosts}
Instead of trying to directly describe a strategy which implies \eqref{eqn:large_explosion_overview}, we first prove
that we can boost the winning probability by $\Theta(1/\sqrt{\ms})$ in $O(\sqrt{\ms}/\omega)$ extra steps. 
More precisely, if a strategy wins with probability $\theta$ after $\m(\theta,n)$ steps,
then we can augment the strategy such that its winning probability is $\theta+\frac{\theta(1-\theta)^3}{4\sqrt{\ms}}$
after $\sqrt{\ms}/\omega$ additional steps. This is the content of Lemma \ref{lem:smallexplosion}:
\begin{lemma}[Small Boost] \label{lem:smallexplosion}
Given constants $0<\theta_1<\theta_2<1$, for any sufficiently large $n\ge 1$ (depending only on $\theta_1,\theta_2$) and any $\theta\in [\theta_1,\theta_2]$, we have that
\[
    \m\lrpar{\theta+\frac{\theta(1-\theta)^3}{4\sqrt{\ms}},n}-\m(\theta,n) \le \sqrt{\ms}/\omega.
    \]
\end{lemma} 
Let us assume that Lemma \ref{lem:smallexplosion} holds for now.
We can then prove \eqref{eqn:large_explosion_overview} by iteratively applying
\Cref{lem:smallexplosion} $\sqrt{\ms}$ times to increase the win probability from $\theta$ to $\theta^*$ in $\sqrt{\ms} \cdot \sqrt{\ms}/\omega = \ms/\omega$ additional steps. We include the details below, and complete the proof of \Cref{stealing}.


\begin{proof}[Proof of Theorem \ref{stealing}]
Let us take $n \ge 1$ sufficiently large (as in Lemma \ref{lem:smallexplosion}) and $\theta\in [\theta_1,\theta_2]$.
Recall that $\theta^* := \min(\theta+\theta(1-\theta)^3/32,\theta_2)$, and we first must show
that \eqref{eqn:large_explosion_overview} holds. I.e., 
\begin{equation} \label{eqn:large_explosion}
\m \left( \min(\theta+\theta(1-\theta)^3/32,\theta_2) ,n \right)-\m(\theta,n) \le \ms/\omega.
\end{equation}
In order to prove this, we iterate Lemma \ref{lem:smallexplosion} $\sqrt{\ms}$ times.
Formally, we define $\gamma_0 :=\theta$, and $\gamma_{i+1}:=\gamma_i+{\gamma_i(1-\gamma_i)^3}/(4\sqrt{\ms})$.
Observe then that by Lemma \ref{lem:smallexplosion}, for each $i\ge 0$, with $\gamma_i\le \theta_2$,
\[
\m(\gamma_{i+1},n)-\m(\gamma_i,n) \le \sqrt{\ms}/\omega.
\]
In particular, if $\gamma_i\le \min(\theta+\theta(1-\theta)^3, \theta_2)\le (1+\theta)/2$, then 
\[ \gamma_{i+1}-\gamma_i\ge \frac{\theta( 1- (1+\theta)/2)^3}{4\sqrt{\ms}}= \frac{\theta(1-\theta)^3}{32\sqrt{\ms}}.\]
Therefore
\[\gamma_{\sqrt{\ms}}\ge \min\left(\theta_2,\theta+\frac{\theta(1-\theta)^3}{32}\right),\]
and so \eqref{eqn:large_explosion} holds. By iterating \eqref{eqn:large_explosion} 
a constant number of times in a similar manner, Theorem \ref{stealing} follows.
\end{proof}

\subsection{Proving Lemma \ref{lem:smallexplosion}}
In this section, we explain the main tools used in the proof of Lemma \ref{lem:smallexplosion}. Fix $\theta \in [\theta_1, \theta_2]$, and set $N:=\m(\theta,n)$ for convenience. Let us suppose that $\scr{S}$ is a strategy which satisfies $\scr{P}$ with probability at least $\theta$ after $N$ steps. 
	First notice that we can assume that $\scr{S}$ is \textit{deterministic} without loss of generality. This is because the optimal strategy for winning in at most $N$ steps deterministically chooses $Y_{i} \in X_{i}$ so as to maximize its win probability (conditional on the current history $X_1, \ldots ,X_{i}$). 
	Since $\scr{S}$ is deterministic, there exists an \textbf{indicator function} $f$ of $\scr{S}$, where $f(X):=1$ if strategy $\scr{S}$ wins when presented the
	edge subsets of $X := (X_1,\dots, X_N)$ in order. To prove Lemma \ref{lem:smallexplosion},
	we augment $\scr{S}$ to get another strategy $\scr{S}'$ which wins with probability at least $\theta^* :=\theta+\frac{\theta(1-\theta)^3}{4\sqrt{\ms}}$
	after $\m(\theta,n) + \sqrt{\ms}/\omega$ steps.


We now give an informal overview of the three main parts to the proof of Lemma \ref{lem:smallexplosion}. The full details appear in the appropriate sections.

\begin{enumerate}
\item `Reducing to the free-move $\scr{D}$-process': We introduce a new game which gives slightly more power to the player
called the \textbf{free-move $\D$-process}. The free-move $\D$-process is played in the same
way as the $\D$-process, except that the player has one opportunity to pick the subset they desire instead of the subset they received (and
then select an edge from this subset).
Since $\scr{P}$ is edge-replaceable, the win probability of any free-move strategy can be matched by a (regular) strategy,
provided the regular strategy is given an additional $\sqrt{\ms}/\omega$ steps (see Lemma \ref{lem:free_move_coupling_short}). Thus it suffices to define a free-move strategy $\scr{F}$ which wins with probability at least $\theta^*$
after $\m(\theta,n)$ steps. 

\item `Defining \ref{alg:free_move_potential}': The free-move strategy \ref{alg:free_move_potential} 
analyzes the Doob-martingale $M=(M_j)_{j=0}^N$ of $f(X)$ with respect to $(X_j)_{j=1}^N$.
Informally, $M_j$ measures the probability that $\scr{S}$ will win, given the first $j$ arriving edge subsets $X_1, \ldots , X_j$.
Based on this interpretation, \ref{alg:free_move_potential} follows the strategy of $\scr{S}$ up until the first time $\tau \ge 1$ 
that there is \textit{potential} to increase its win probability. 
\ES{In particular, there is an edge subset $W_\tau$ such that replacing the edge set $X_\tau$ with $W_\tau$ increases the probability that $\scr{S}$ will win by at least $c$.}
At this point, it invokes its free-move
to swap $X_{\tau}$ with $W_{\tau}$, and then follows the strategy $\scr{S}$ as if $X_1, \ldots , W_{\tau}$
where the first $\tau$ subsets to arrive. Conditional on $\tau \le N$, this guarantees that \ref{alg:free_move_potential} has
a win probability at least $c$ greater than $\scr{S}$.

\item `Bounding the win probability of \ref{alg:free_move_potential}': 
In order to prove that \ref{alg:free_move_potential} attains a win probability significantly
better than $\scr{S}$, we must prove that $\mb{P}[\tau \le N] = \Omega(1)$.
We do so by proving a martingale concentration result (Theorem \ref{typical}),
and then applying it in a non-standard way. Observe that the function $f$ is $\{0,1\}$-valued, and so since
$0 < \theta < 1$, $f(X)$ \textit{cannot} be concentrated about $\mb{E}[f(X)]$. On the other hand,
we argue that if $\mb{P}[\tau \le N] = o(1)$, then Theorem \ref{typical} would force
$f(X)$ to be concentrated. Thus, we can conclude that $\mb{P}[ \tau \le N] = \Omega(1)$. 

\end{enumerate}

\subsubsection{The Free-Move $\scr{D}$-Process}	\label{sec:free_move}
The \textbf{free-move $\D$-process} is defined in the same
way as the $\D$-process, except that the player can adaptively choose a time $\tau \ge 1$, such that if $X_1, \ldots , X_{\tau}$
were the previously presented subsets of edges,
then they can choose an arbitrary subset $W_{\tau}$
from $\supp(\D)$ (the support of $\D$). They then
get to add an edge $Y_{\tau} \in W_{\tau}$ to $G_{\tau-1}$, opposed to an edge
from $X_{\tau}$ (as in the standard game).

Clearly, any strategy for the standard $\D$-process is a strategy
for the free-move $\D$-process. Thus, satisfying an edge-monotone property $\scr{P}$
in the latter game is no harder than in the former game. However, if $\scr{P}$ is edge-replaceable then the advantage gained by the player is
not very significant, and so this new game is a good approximation of the original game. We extend all the definitions from the standard $\D$-process to formalize
this intuition. Specifically, if $\scr{F}$ is a strategy for the free-move $\D$-process,
then $G^{\scr{F}}_t$ is the graph constructed by following $\scr{F}$ in the first $t$ steps. Moreover, $T_{\scr{F}}$ is defined to be the first $t \ge 1$ such that $G^{\scr{F}}_t \in \scr{P}$
(where $T_{\scr{F}} := \infty$ if no such $t$ exists.)

\begin{lemma} \label{lem:free_move_coupling_short}
Let $\scr{F}$ be a strategy for the free-move $\D$-process process
for satisfying a property $\scr{P}$ which is $\omega$-edge-replaceable.
In this case, there exists a strategy $\scr{F}'$ for the (standard) $\D$-process, 
such that for each $k \ge 1$,
    $\mb{P}[ G^{\scr{F}'}_{k + \sqrt{\ms}/\omega} \in \scr{P}] \ge \left(1 - o(1/\sqrt{\ms}) \right) \cdot \mb{P}[ G^{\scr{F}}_k \in \scr{P}]$.
\end{lemma}

\begin{proof}[Proof of Lemma \ref{lem:free_move_coupling_short}]
Let us assume that $\scr{P}$ is $\omega$-edge-replaceable,
and $\scr{F}$ is a strategy for the free-move $\D$-process process. In order to prove
the lemma, it suffices to show that there exists a strategy $\scr{F}'$ for the (standard) $\D$-process, such that if both strategies are presented the same (random) edge subsets $(X_t)_{t=1}^{\infty}$, then with probability $1 - o(1/\sqrt{\ms})$ we have that
\begin{equation}\label{eqn:coupling_edge_replacement}
    T_{\scr{F}'} \le T_{\scr{F}} + \sqrt{\ms}/\omega.
\end{equation}
We begin by defining $\scr{F}'$ to follow the same decisions of $\scr{F}$ up until time 
$T_{\scr{F}}$, where if $\scr{F}$ invokes a free-move at some time $1 \le \tau \le T_{\scr{F}}$, 
then we define $\scr{F}'$ to choose an edge of $X_{\tau}$ arbitrarily. If $\scr{F}$ 
does not invoke a free-move, then $\tau := \infty$, and the strategies execute identically.

Let $G_{T_{\scr{F}}}$ and $G'_{T_{\scr{F}}}$ be the graphs constructed by $\scr{F}$ and $\scr{F}'$ after $T_{\scr{F}}$ steps, respectively. At this point, $G_{T_{\scr{F}}} \in \scr{P}$ (by definition of $T_{\scr{F}}$), yet $G'_{T_{\scr{F}}}$ may not satisfy $\scr{P}$. 
Specifically,
if $\tau < \infty$, then $G'_{T_{\scr{F}}}$ will be missing the edge $e$ that $\scr{F}$ added at step $\tau$. Note that $G'_{T_{\scr{F}}}+e\in \scr{P}$, so after $T_{\scr{F}}$ steps we define $\scr{F}'$ to run the edge-replacement procedure of $\scr{P}$ to ensure that after another $\sqrt{\ms}/\omega$ steps, it will be left with a graph which satisfies $\scr{P}$ with probability $1 - o(1/\sqrt{\ms})$. This completes the proof of \eqref{eqn:coupling_edge_replacement}, and so the lemma is proven.

\end{proof}

\subsubsection{Defining \ref{alg:free_move_potential}}


Recall that $\scr{S}$ is a deterministic strategy
which wins with probability at least $\theta$ after $N:=N(\theta)=\m(\theta,n)$ steps,
and $f$ is its indicator function. Observe
that $\mu :=\mb{E}[f(X)] =\mb{P}[f(X) =1] \ge \theta$ for $X =(X_1, \ldots , X_N)$,
where each $X_i$ is drawn independently from $\scr{D}$. 
Setting $C(\theta):=1+\log_2\left( \frac{1}{1-\theta}\right)$, we define
\begin{equation}\label{eqn:potential_amount}
c:=\frac{\mu(1- \mu)}{\sqrt{2C(\theta)\ms}}.
\end{equation}
The dependence of $c$ on $\theta$ and $\mu$ is for technical reasons which will only
become relevant in Section \ref{sec:analysis_gain}. For now, it suffices to think of $c$ as $\Theta(1/ \sqrt{\ms})$.
Our goal is to identify instantiations of $X$ in which by using the \textit{free-move}
of \ref{alg:free_move_potential}, we can boost the win probability of $\scr{S}$ by $c$.

We first consider the Doob-martingale $M =(M_j)_{j=0}^N$ of $f(X)$ 
with respect to $(X_j)_{j=1}^N$. That is, $M_0 := \mb{E}[f(X)]$ and $M_j := \mb{E}[ f(X) \mid X_1, \ldots X_j]$
for $j \in [N]$, where $M_N = f(X)$. Moreover, for each $1 \le j \le N$, define the function $f_j$,
where for each $(r_1, \ldots ,r_j) \in \supp(\D)^j$,
\begin{equation} \label{eqn:martingale_function}
    f_{j}(r_1, \ldots ,r_{j}):= \bE [ f(X) \mid (X_i)_{i=1}^j=(r_i)_{i=1}^j].
\end{equation}
Equivalently, $f_j(r_1, \ldots ,r_{j})$ is the probability that $\scr{S}$ wins after
$N$ steps, conditional on $X_{1} =r_1,\ldots ,X_{j} =r_j$. Observe that 
$M_j =f_j(X_1, \ldots ,X_j)$ by construction.
We say that $(r_1, \ldots ,r_j) \in \supp(\D)^j$ has \textbf{potential}, provided there exists $w_j \in \supp(\D)$ such that
\begin{equation}\label{eqn:potential}
f_j(r_1, \ldots ,r_j) +c < f_j(r_1, \ldots ,w_j).
\end{equation}
In this case, we refer to $w_j$ as a \textbf{witness} for $(r_1, \ldots ,r_j)$.
Note that there may be multiple witnesses for $(r_1, \ldots ,r_j)$. 
Intuitively, if $(r_1, \ldots ,r_j)$ has potential, then
$\scr{S}$ has a better win probability when $X_1 =r_1, \ldots ,X_j = w_j$, opposed to when $X_1 =r_1, \ldots ,X_j = r_j$,  While we cannot ensure that $X_j = w_j$ in the standard $\D$-process, we \textit{can} in the free-move $\D$-process.

Algorithm \ref{alg:free_move_potential} runs for $N$ steps, and yet has a slightly higher win probability
than $\scr{S}$. We assume that the algorithm is presented the subsets $X_1, \ldots , X_N$ in order.
We choose the edges in the same way as $\scr{S}$ up until the first step $1 \le t \le N$ such that $(X_1, \ldots , X_{t})$ has potential. 
Let us define $1 \le \tau \le N$ to be this step, where $\tau := \infty$ if no such step exists.
Assuming $\tau \le N$, we identify an arbitrary witness $W_{\tau}$ of $(X_1, \ldots , X_{\tau})$. At this point,
we invoke our free-move, and replace $X_{\tau}$ with $W_{\tau}$. For step $\tau$ and each subsequent step, we choose the edges
by following the strategy of $\scr{S}$ with $X_{\tau}$ replaced by $W_{\tau}$. 
Below is a formal description of the algorithm: 

\newpage

\begin{varalgorithm}{$\mathtt{PotentialBoost}$}
\caption{Free-move Strategy}
\label{alg:free_move_potential}
\begin{algorithmic}[1]
\Require $\til{G}_{0}=([n], \emptyset)$.
\Ensure $\til{G}_{N}$ 
\For{$t=1, \ldots , \min\{\tau -1, N\}$}          \Comment{follow decisions of $\scr{S}$}
\State Define $Y_t$ to be the edge chosen by $\scr{S}$ when given $(X_1, \ldots ,X_t)$.   
\State $\til{G}_{t} := \til{G}_{t-1} \cup Y_t$.
\EndFor
\If{$\tau \le N$}  let $W_{\tau}$ be an arbitrary witness of $(X_1, \ldots , X_{\tau})$.          \Comment{$(X_1, \ldots , X_{\tau})$ has potential}
\State Define $Y_{\tau}$ be the edge chosen by $\scr{S}$ when given $(X_1, \ldots ,W_{\tau})$.
\State $\til{G}_{\tau} := \til{G}_{\tau-1} \cup Y_\tau$.          \Comment{execute a free-move}
\For{$t= \tau+1, \ldots , N$}                                             \Comment{follow $\scr{S}$ with $X_{\tau}$ replaced with $W_{\tau}$}
\State Define $Y_t$ be the edge chosen by $\scr{S}$ when given $(X_1, \ldots ,W_{\tau}, \ldots , X_t)$.
\State $\til{G}_{t} := \til{G}_{t-1} \cup Y_t$.
\EndFor
\EndIf
\State \Return $\til{G}_{N}$.
\end{algorithmic}
\end{varalgorithm}
Let $G_N = G_{N}^{\scr{S}}$ be the graph formed by $\scr{S}$ when passed edge subsets $X_1, \ldots , X_N$ \ES{and let $\til{G}_N$ be the corresponding output from the \ref{alg:free_move_potential} algorithm}.
We compare $\til{G}_N$ to $G_N$: 
\begin{lemma} \label{lem:free_move_guarantee}
The graph $\til{G}_N$ satisfies the following:
\begin{enumerate}
    \item If $\mb{P}[ \tau > N] > 0$,
    then $\mb{P}[ \til{G}_N  \in \scr{P} \mid \tau > N] = \mb{P}[G_{N} \in \scr{P} \mid \tau > N].$ \label{item:no_free_move}
    \item If $\mb{P}[ \tau \le N] > 0$, then $\mb{P}[ \til{G}_{N} \in \scr{P} \mid \tau \le N] > \mb{P}[ G_{N} \in \scr{P} \mid \tau \le N]  + c$  \label{item:free_move}
    \item $\mb{P}[ \til{G}_N  \in \scr{P}] \ge \mb{P}[ G_{N} \in \scr{P} ] + c \cdot \mb{P}[ \tau \le N]$. \label{item:overall_boost}
\end{enumerate}
\end{lemma}

\begin{proof}
We prove the properties of Lemma \ref{lem:free_move_guarantee} in order. First observe that $\tau \le N$ if and only if \ref{alg:free_move_potential} makes a free-move at
some step. Moreover, if \ref{alg:free_move_potential} does \textit{not} make a free move, then the algorithm
simply executes $\scr{S}$ as the subsets $X_1, \ldots , X_N$ arrive. Thus, $\til{G}_N$ and
$G_N$ are the same graph, and so in particular,
\[
    \mb{P}[ \til{G}_N  \in \scr{P} \mid \tau > N] = \mb{P}[ G_{N} \in \scr{P} \mid \tau > N].
\]

Let us now consider the case $\tau \le N$. It will be convenient to define $R$ to be those $(r_1, \ldots ,r_k) \in \cup_{i=1}^{N} \supp(\D)^{i}$,
such that $(r_1, \ldots, r_k)$ has potential, yet no proper prefix of $(r_1, \ldots ,r_k)$ has potential. Observe
that conditional on $\tau \le N$, $(X_1, \ldots , X_{\tau})$ is supported on $R$.
Now, fix $(r_1, \ldots , r_k) \in R$, and condition on $(X_1, \ldots , X_{k}) = (r_1, \ldots ,r_k)$.
Observe then that $\til{G}_N$ is distributed as $G_N$ conditional on $(X_1, \ldots , X_{k}) = (r_1, \ldots , w_k)$.
Thus, for each $(r_1, \ldots , r_{k}) \in R$,
\begin{align*}
    \mb{P}[\til{G}_N \in \scr{P} \mid (X_i)_{i=1}^k=(r_i)_{i=1}^k] &= \mb{P}[G_N \in \scr{P} \mid (X_i)_{i=1}^{k-1}=(r_i)_{i=1}^{k-1}, X_k=w_k] \\
    &=\mb{E}[f(X) \mid (X_i)_{i=1}^{k-1}=(r_i)_{i=1}^{k-1}, X_k=w_k] \\
    &= f_{k}(r_1, \ldots ,w_k)> f_{k}(r_1, \ldots ,r_k) +c,
\end{align*}
where \CM{the} second equality uses the definition of $f$, and the final inequality holds since $(r_1, \ldots ,r_k)$ has potential. 
By averaging over all the elements of $R$, property \eqref{item:free_move} follows.
Property \eqref{item:overall_boost} is implied by \eqref{item:no_free_move} and \eqref{item:free_move}:
\[
    \mb{P}[ \til{G}_N  \in \scr{P}] \ge \mb{P}[\text{$G_{N} \in \scr{P}$, $\tau > N$}]   + \mb{P}[\text{$G_{N} \in \scr{P}$, $\tau \le N$}]  + c \cdot \mb{P}[ \tau \le N] 
    = \mb{P}[ G_{N} \in \scr{P} ] + c \cdot \mb{P}[ \tau \le N].
\]
\end{proof}

\subsubsection{Bounding the Win Probability of \ref{alg:free_move_potential}} \label{sec:analysis_gain}
Observe that property \eqref{item:overall_boost} of Lemma \ref{lem:free_move_guarantee} ensures
\ref{alg:free_move_potential} has a win probability at least as large as $\scr{S}$. Moreover, by definition,
$c = \Theta(1/\sqrt{\ms})$. Thus, if we can show that the stopping time $\tau$
of \ref{alg:free_move_potential} satisfies $\mb{P}[\tau \le N] = \Omega(1)$,
then this will prove that \ref{alg:free_move_potential} boosts the win probability of $\scr{S}$
by $\Omega(1/\sqrt{\ms})$, as (roughly) claimed by Lemma \ref{lem:smallexplosion}.
Before proceeding with this lower bound, we state the following upper bound on $N(\theta)$, which relies on a standard multi-round exposure argument
to boost the win probability from $1/2$ to $\theta$.
\begin{proposition} \label{prop:multi_round}
If $C(\theta)=1+\log_2\left( \frac{1}{1-\theta}\right)$, then $N(\theta) \le C(\theta) \ms.$
\end{proposition}

\begin{proof}
	Suppose $\cS'$ is a strategy that succeed with probability at least $1/2$ in $\m$ steps. For any integer $k$, consider a strategy that runs for $k \m$ steps where for any $i=0,\dots, k-1$, in steps $\{i\m+1, i\m+2,\dots, (i+1)\m\}$ we run the strategy $\cS'$ as if the graph is empty. Then the probability of failure after $k N(1/2)$ steps is at most $(1/2)^k$. By letting $k=\lceil\log_2 (1/(1-\theta))\rceil$ and noting that $1-(1/2)^k\ge \theta$, we get that 
	\[N(\theta)\le k \m\le \left(1+\log_2\left( \frac{1}{1-\theta}\right)\right) \m\] 
	as desired.
\end{proof}

\begin{lemma} \label{lem:quantify_boost}
If $\mu := \mb{P}[ G_{N} \in \scr{P}]$, $\Pr[ \tau \le N] \ge  \frac{1-\mu}{2}$.
\end{lemma}
To establish Lemma \ref{lem:quantify_boost}, we invoke a concentration inequality for the Doob martingale
of $f(X)$ with respect to $(X_j)_{j=1}^N$ (see Corollary \ref{simpletypical}). 
We state and prove the full theorem in Section \ref{sec:approx_balanced},
and for now just indicate how we apply a special case of this theorem for our specific needs. 
The rough idea is as follows. If $\mb{P}[\tau \le N]$ were $o(1)$, then our concentration inequality
would imply that $f(X)$ must be concentrated about its expectation. But $f(X) \in \{0,1\}$, and $\mb{E}[f(X)] = \mu \ge \theta$,
so since we may assume that $\mu$ is bounded away from $1$, this is not possible. Thus, $\mb{P}[ \tau \le N]$ must be $\Omega(1)$.

To formalize this intuition, let us say that $r=(r_1, \ldots ,r_N) \in \supp(\D)^N$ is \textbf{stable} if no prefix
of $r$ has potential. That is, for each $1 \le j \le N$ and $w_j \in \supp(\scr{D})$,
\begin{equation}\label{eqn:stable_element}
f_{j}(r_1, \ldots ,w_j) - f_{j}(r_1, \ldots ,r_j) \le c.
\end{equation}
Define $\Gamma\subseteq \supp(\D)^N$ to be the stable elements of $\supp(\D)^N$. We relate $\Gamma$ to the stopping time $\tau$ of \ref{alg:free_move_potential}
in the following way:
\begin{proposition} \label{prop:gamma_tau_equiv}
$X=(X_1, \ldots , X_N) \in \Gamma$ if and only if $\tau > N$. In particular, 
$\Pr[X\not\in \Gamma] = \Pr[ \tau \le N]$.
\end{proposition}
We then invoke the following one-sided concentration inequality 
to lower bound $\Pr[X \not \in \Gamma]$:
\begin{cor}[of Theorem \ref{typical}]\label{simpletypical} 
	For each $t \ge 0$,
		$\Pr[f(X)\le \bE f(X)-t]\le \exp\left(\frac{-2t^2}{Nc^2}\right)+\Pr[X\not\in \Gamma]. $
\end{cor}
\begin{proof}[Proof of Lemma \ref{lem:quantify_boost}]
By setting $t= \mu/2$ where $\mu = \mb{P}[f(X) =1] = \mb{E}[f(X)]$, Corollary \ref{simpletypical} implies that
\[
1-\mu=\Pr[f(X)=0]=\Pr[f(X)\le \mu/2 ]\le  \exp\left(\frac{-\mu^2}{2N c^2}\right)+\Pr[X \not\in \Gamma].
\]
Thus, since $c:=\frac{\mu(1- \mu)}{\sqrt{2C(\theta)\ms}}$, and $N \le C(\theta) \ms$ by Proposition \ref{prop:multi_round},
we get that $\mu^2/(2Nc^2)\ge (1-\mu)^{-2}$. Now, $\Pr[X\not\in \Gamma] = \Pr[ \tau \le N]$, by Proposition \ref{prop:gamma_tau_equiv}, so it follows that
\[
 \Pr[ \tau \le N] \ge 1-\mu-\exp\left(\frac{-1}{(1-\mu)^2}\right) \ge \frac{1-\mu}{2},
\]
where the last step uses the elementary inequality $ \exp({-1/z^2})\le z/2$ for $z\in (0,1)$.
\end{proof}	
	

\subsubsection{Putting it All Together}

\begin{proof}[Proof of Lemma \ref{lem:smallexplosion}]
Let us set $N' := N + \sqrt{\ms}/\omega$ for convenience. Observe that by Lemma \ref{lem:free_move_coupling_short},
we are guaranteed a strategy for the standard $\D$-process which constructs $G_{N'}$ such that
\[
\mb{P}[G_{N'} \in \scr{P}] \ge \left(1 - o(1/\sqrt{\ms})\right) \cdot \mb{P}[\til{G}_N \in \scr{P}].
\]
Now, after applying Lemmas \ref{lem:free_move_guarantee} and \ref{lem:quantify_boost},
we get that
$\mb{P}[\til{G}_N \in \scr{P}] \ge \mu  + \frac{\mu (1-\mu)^2}{\sqrt{8 C(\theta)\ms}},$
for $\mu = \mb{P}[ G_{N} \in \scr{P}]$.
On the other hand, $\mb{P}[ G_{N} \in \scr{P} ] \ge \theta$, and $z \rightarrow z + \frac{z (1-z)^2}{\sqrt{8 C(\theta)\ms}}$
is increasing
as a function of $z$ (it is routine to check that it has a positive derivative), so we get that
\[
\mb{P}[ \til{G}_N  \in \scr{P}] \ge \theta  + \frac{\theta (1-\theta)^2}{\sqrt{8 C(\theta)\ms}}.
\]
However, $C(\theta):=1+\log_2\left( \frac{1}{1-\theta}\right)$, so $C(\theta)\le 1/(1-\theta)^2$ by the elementary inequality $1+\log_2(z)\le z^2$ for $z\ge 1$.
Thus, $\mb{P}[ \til{G}_N  \in \scr{P}] \ge \frac{\theta (1-\theta)^3}{\sqrt{8\ms}}$,
and so
\begin{align*}
    \mb{P}[G_{N'} \in \scr{P}] &\ge \left(1 - o(1/\sqrt{\ms}) \right) \left(\theta + \frac{\theta (1-\theta)^3}{ \sqrt{8 \ms}} \right)\ge \theta + \frac{\theta (1-\theta)^3}{ 4\sqrt{\ms}},
\end{align*}
where the last inequality holds for sufficiently large $n$ (dependent on $\theta_1$ and $\theta_2$).
\end{proof}




\section{On Approximately Balanced Martingales} \label{sec:approx_balanced}
Let $S_0,\dots, S_k$ be finite sets,
and suppose that $X=(X_j)_{j=0}^{k}$ is a random variable in $S:=S_0\times \dots \times S_k$,
where $S_j := \supp(X_j)$.
Moreover, assume that $M=(M_j)_{j=0}^{k}$ is a martingale with respect to $(X_j)_{j=0}^{k}$. Thus,
there exists a function $m_{j}: S_0\times \dots \times S_j \rightarrow \mb{R}$, such
that $M_j = m_{j}(X_0, \ldots ,X_j)$. 
Given a constant $c_j \ge 0$, we say that $M_j$ is \textbf{balanced} (with respect to $c_j$), provided for all $(s_0, \ldots ,s_j) \in S_0 \times \dots \times S_j$ and $s_{j}' \in S_j$,
\begin{equation} \label{eqn:balanced_definition}
    m_{j}(s_0, \ldots ,s'_j) - m_{j}(s_0, \ldots ,s_j) \le c_j.
\end{equation}
From the definition of martingale, we get the following:
\begin{proposition} \label{prop:balanced_implies_bounded}
If $M_j$ is balanced, then $|M_j - M_{j-1}| \le c_j$.
\end{proposition}
If we are given constants $c =(c_j)_{j=1}^k$, such that each $M_j$
is balanced with respect to $c_j$, then we say that $M$ is \textbf{balanced} (with respect to $c$).
Observe that if $M$ is balanced, then $|M_j - M_{j-1}| \le c_j$ for all $j \in [k]$ (i.e., 
$M$ is $c$-\textbf{Lipschitz}).  As a result, one can apply the Azuma-Hoeffding inequality to argue that $M_k$ is concentrated about $M_0$.
On the other hand, if $M$ is $c$-Lipschitz, then $M$ is $(2c_j)_{j=0}^k$ balanced. Thus, the balanced
property is also \textit{necessary} to apply the Azuma-Hoeffding inequality.

This raises the question of what can be done if $M$ is not balanced.
We provide a lower tail concentration inequality which depends on the probability
\textit{each} $M_j$ satisfies \eqref{eqn:balanced_definition}
on the randomly chosen point $(X_0, \ldots ,X_{j-1},X_j)$, for all $s'_j \in S_j$. More formally, we say that $(s_0, \ldots ,s_k) \in S$ is \textbf{stable}
with respect to $M$ and $c$, provided for all $1 \le j \le k$ and $s_{j}' \in S_j$,
\begin{equation} \label{eqn:gamma_definition}
    m_{j}(s_0, \ldots ,s'_j) - m_{j}(s_0, \ldots ,s_j) \le c_j.
\end{equation}
Define $\Gamma_M \subseteq S$ to be the stable elements of $S$.
We measure the balance of $M$ based on the value of $\Pr[X \in \Gamma_M]$,
where $\Pr[X \in \Gamma_M] =1$ indicates that $M$ perfectly
satisfies the balanced definition.

\begin{thm}\label{typical}
Suppose $M=(M_j)_{j=0}^k$ is martingale with respect to a sequence of discrete random variables
$X=(X_j)_{j=0}^{k}$ in $S=S_0 \times \dots \times S_k$, where $S_j := \supp(X_j)$. Given constants $c = (c_j)_{j=1}^k$, let $\Gamma_M \subseteq S$ be the stable elements of $S$ with respect to $M$ and $c$.
In this case, for any $t \ge 0$, 
\[
\Pr[M_k \le M_0 - t]\le \exp\left(\frac{-2 t^2}{\sum_{j= 1}^{k} c_j^2}\right)+\Pr[ X\not\in \Gamma_M].
\]	\end{thm}
\begin{remark}
We can derive an upper tail concentration inequality by negating the left-hand side of
\eqref{eqn:gamma_definition} to modify the definition of $\Gamma_M$.
We also note that our approach can be seen as a refinement of the decision tree approach of \cite{chung2006concentration}, which was used to prove various concentration inequalities for martingales which are tolerant to ``bad" events.
\end{remark}
In order to prove Theorem \ref{typical},
we couple $M = (M_j)_{j=0}^k$ with another martingale $M' = (M'_j)_{j=0}^k$
which is balanced and dominated by $M$ on $\Gamma_M$. 
\begin{lemma}\label{lem:martingale_coupling}
There exists a coupling of $M$, and another martingale $M' = (M'_j)_{j=0}^{k}$ with respect
to $(X_j)_{j=0}^k$, such that the following conditions hold:
\begin{enumerate}[label=(\subscript{Q}{{\arabic*}})]
    \item `Initial values': $M'_0 = M_0$.  \label{prop:initial_values}
    \item `Balanced': $M'$ is balanced with respect to $c_1, \ldots , c_k$. \label{prop:balanced_increments}
    \item `Domination': If $X \in \Gamma_M$, then $M'_{j} \le M_j$ \label{prop:domination}
    for all $j \in [k]$. 
\end{enumerate}
\end{lemma}
\begin{proof}
In order to prove the lemma for $M= (M_j)_{j=0}^k$,
we proceed inductively on the value $k$.
Firstly, observe that if $k=0$, then we may set $M_{0}':=M_0$, and so the required properties
hold trivially. Let us now take $k \ge 1$, and assume that the lemma holds for $k-1$.

In order to simplify the notation below, let $\scr{F}_{j} = \sigma(X_0, \ldots ,X_j)$ be the sigma-algebra
generated by $X_0, \ldots , X_j$. It will be convenient to first assume that
$X_0$ is constant, so that $\scr{F}_0$ is the trivial sigma-algebra (i.e., $\scr{F}_0 = \{\emptyset, \Omega\}$), and $M_0$ is constant.  \CM{Note that then the function $m_j$ satisfies
$M_{j} = m_{j}(X_1, \ldots , X_j)$ for $1 \le j \le k$.}

As in the base case, we first set $M'_{0} := M_0$ so that \ref{prop:initial_values} is satisfied. In order to define $M'_1$,
the high level idea is to \ES{modify} $M_1$ in such a way that $M'_1$ is balanced,
while maintaining the martingale property. \ES{This requires us to shift each element's value either up or down.
In order to also satisfy the domination property, we must ensure
that certain elements are only ever downshifted.}

Let us say that $s_1 \in S_1$ is \textit{small}, provided $m_1(s_1) < m_1(s'_1) - c_1$
for some $s'_1 \in S_1$. Let $A$ be the small elements of $S_1$,
and $B:= S_1 \setminus A$.
Observe that if $A \neq \emptyset$, then $B \neq \emptyset$.
Moreover, for each $b, b' \in B$, we have that
\begin{equation} \label{eqn:balanced}
    |m_{1}(b) - m_{1}(b')| \le c_1. 
\end{equation}
We refer to $B$ as the \textit{large} elements of $S_1$.
Let us proceed with our construction under the assumption that $A \neq \emptyset$,
so that $\Pr[X_1 \in A] > 0$ and $\Pr[X_1 \in B] >0$.
When $A = \emptyset$, the construction follows easily
from the inductive assumption.

Observe that $\bE[ M_1 \mid X_1 \in A]\le  \max_{b \in B} m_{1}(b)$ \ES{since $m_1(a)\le m_1(b)$ for all $a\in A, b\in B$}.
Thus, there exists $\gamma \ge 0$ such that
\CM{\begin{equation*}
    \bE[ M_1  \mid X_1 \in A] + \frac{\gamma}{\mb{P}[X_1 \in A]} + \frac{\gamma}{\Pr[X_1 \in B]} \in \left[\min_{b \in B}m_{1}(b), \max_{b \in B} m_{1}(b)\right],
\end{equation*}
and so}
\begin{equation} \label{eqn:shift_gamma}
    \bE[ M_1  \mid X_1 \in A] + \frac{\gamma}{\mb{P}[X_1 \in A]} \in \left[\min_{b \in B}m_{1}(b) - \frac{\gamma}{\Pr[X_1 \in B]}, \max_{b \in B} m_{1}(b) - \frac{\gamma} {\Pr[X_1 \in B]}\right].
\end{equation}
Setting $\gamma_A := \gamma/\Pr[X_1 \in A]$ and $\gamma_B := \gamma/ \Pr[X_1 \in B]$ for convenience, we
define
\begin{equation} \label{eqn:level_one_martingale}
    M'_1 := (\bE[ M_1  \mid X_1 \in A] + \gamma_A) \cdot \bm{1}_{[X_1 \in A]} + (M_1 - \gamma_B) \cdot \bm{1}_{[X_1 \in B]}.
\end{equation}
Thus, relative to $M_1$, $M'_1$ lowers the value of each $b \in B$ by $\gamma_B$,
and assigns $\bE[ M_1  \mid X_1 \in A] + \gamma_A$ to every $a \in A$.
Observe first that because of \eqref{eqn:balanced} and \eqref{eqn:shift_gamma}, we have that $|m'_{1}(s_1) - m'_{1}(s'_1)| \le c_1$
for each $s_1,s'_1 \in S_1$. 
In addition, observe that
\begin{align*}
\bE[M'_{1} \mid \scr{F}_0] &=  \bE[ \, (\bE[ M_1 \mid X_1 \in A] +  \gamma_A)\cdot \bm{1}_{[X_1 \in A]}] + \mb{E}[ (M_1 - \gamma_B) \cdot \bm{1}_{[X_1 \in B]}] \\
                           &=  \bE[M_1 \cdot \bm{1}_{[X_1 \in A]}] + \gamma_A \mb{P}[ X_1 \in A]  + \mb{E}[M_1 \cdot \bm{1}_{[X_1 \in B]}] - \gamma_B \mb{P}[X_1 \in B]\\
                           &= \mb{E}[M_1 \cdot (\bm{1}_{[X_1 \in A]} + \bm{1}_{[X_1 \in B]})] + \gamma - \gamma \\
                           &= \mb{E}[M_1] = M_0,
\end{align*}
where the last line follows from the martingale property of $M_1$.
Thus, $\bE[M'_{1} \mid \scr{F}_0] = M'_0$, and so $M'_1$ also satisfies the martingale property.

We now construct $(M'_j)_{j=2}^k$ and verify the remaining properties. For each $j \in [k]$, let
\[
Y_j := (\bE[ M_1  \mid X_1 \in A] + \gamma_A) \cdot \bm{1}_{[X_1 \in A]} + (M_j - \gamma_B) \cdot \bm{1}_{[X_1 \in B]}.
\]
(Note that $Y_1 = M'_1$). We claim that $Y = (Y_j)_{j=1}^k$ is a martingale with respect to $(X_j)_{j=1}^k$.
In order to see this, fix $2 \le j \le k$, and take the conditional expectation with respect
to $\scr{F}_{j-1}$:
\begin{align*}
    \bE[ Y_j \mid \scr{F}_{j-1}] &= \bE[ (M_j - \gamma_B) \cdot \bm{1}_{[X_1 \in B]} \mid \scr{F}_{j-1}] + \bE[ \, (\bE[ M_1 \mid X_1 \in A] + \gamma_A)\cdot \bm{1}_{[X_1 \in A]} \mid \scr{F}_{j-1}] \\
                                &= \bE[ (M_j - \gamma_B)  \mid \scr{F}_{j-1}] \cdot \bm{1}_{[X_1 \in B]} +  (\bE[ M_1 \mid X_1 \in A] + \gamma_A) \cdot \bm{1}_{[X_1 \in A]} \\
                                &= (M_{j-1} - \gamma_B)  \cdot \bm{1}_{[X_1 \in B]} + (\bE[ M_1 \mid X_1 \in A] + \gamma_A) \cdot \bm{1}_{[X_1 \in A]} =:Y_{j-1}.
\end{align*}
The first equality follows since the random variables $\bm{1}_{[X_1 \in B]}, \bm{1}_{[X_1 \in A]}$ and $\bE[ M_1 + \gamma_A \mid X_1 \in A]$ are determined
by $X_1, \ldots , X_{j-1}$ (and thus can be viewed as constants), and the second uses the martingale property of $M$.

Let $\Gamma_Y$ be the stable elements of $S_1 \times \dots \times S_k$ with respect
to $Y$ and $c$. By applying the inductive assumption to $Y$, we get a martingale which can be coupled
with $Y$, and whose initial term is $Y_1$. Since $Y_1 = M'_1$, we can denote this martingale unambiguously 
by $(M'_j)_{j=1}^k$. Observe that it has the following properties:
\begin{enumerate}
    \item $(M'_j)_{j=1}^k$ is balanced with respect to $c_2,\dots, c_k$. \label{prop:balanced_increments_ind}
    \item If $(X_1, \ldots , X_k) \in \Gamma_Y$, then $M'_j \le Y_j$ for $j=1, \ldots ,k$. \label{prop:domination_ind}
\end{enumerate}
We claim that $M' = (M'_j)_{j=0}^k$ is a martingale which satisfies properties \ref{prop:initial_values}, \ref{prop:balanced_increments},
and \ref{prop:domination}. We prove these statements in order.

We have already verified that $\bE[M'_1 \mid \scr{F}_0] = M'_{0}$.
Moreover, $(M'_j)_{j=1}^k$ satisfies the martingale property by the inductive assumption.
Thus, $M' = (M'_j)_{j=0}^k$ is a martingale with respect to $(X_j)_{j=0}^k$.

By construction, $M'_0 = M_0$, and so \ref{prop:initial_values} holds.
Now, $M'_2, \ldots M'_k$ are balanced by \eqref{prop:balanced_increments_ind}, 
and we have already verified that $M'_1$ is balanced.
Thus, $M'$ satisfies \ref{prop:balanced_increments}.
It remains to verify \ref{prop:domination}.

\CM{We shall first show that if $(X_0,X_1,, \ldots ,X_k) \in \Gamma_M$, then $(X_1, \ldots ,X_k) \in \Gamma_Y$.
Now, by the definition of $\Gamma_M$, we have that for each $1 \le j \le k$ and $s'_j \in \supp(S_j)$, 
\begin{equation} \label{eqn:gamma_m}
    m_{j}(X_1, \ldots ,X_{j-1},s'_j) - m_{j}(X_1, \ldots ,X_j) \le c_j.
\end{equation}
It suffices to show that $y_{j}(X_1, \ldots ,X_{j-1},s'_j) - y_{j}(X_1, \ldots ,X_j) \le c_j$,
where $y_{j}$ is the function which satisfies $y_{j}(X_1, \ldots ,X_j) = Y_j$.
Observe that if $X_1 \in A$, then $y_{j}(X_1, \ldots ,X_{j-1},s'_j) = y_{j}(X_1, \ldots ,X_j)$. Otherwise, if $X_1 \in B$,
then
\[
y_{j}(X_1, \ldots ,X_{j-1},s'_j) - y_{j}(X_1, \ldots ,X_j) = m_{j}(X_1, \ldots ,X_{j-1}, s'_j) - m_{j}(X_1, \ldots ,X_j) \le c_j,
\]
where the inequality follows from \eqref{eqn:gamma_m}. Thus, $(X_1, \ldots ,X_k) \in \Gamma_Y$, and so
by the inductive assumption, $M'_j \le Y_j$ for $j=1, \ldots ,k$. On the other hand,
since $(X_1, \ldots, X_k) \in \Gamma_M$, $m_{1}(X_1) \ge m_{1}(b') - c_1$ for all $b' \in S_1$.
Thus, $X_1$ is large (i.e., $X_1 \in B$). It follows
that $Y_j = M_j - \gamma_B \le M_j$, and so $M'_j \le M_j$ for $j=1, \ldots ,k$,
which proves that \ref{prop:domination} holds.}

\CM{
To complete the inductive step, we must handle the case when $X_0$ is not necessarily constant. We can handle
this by applying the above martingale construction to $(m_{j}(s_0, X_1 \ldots ,X_j))_{j=0}^{k}$
for each $s_0 \in \supp(X_0)$. The proof is thus complete.
}

\end{proof}
\begin{proof}[Proof of Theorem \ref{typical}]
Fix $t \ge 0$, and let $M' = (M'_j)_{j=0}^k$ be the martingale with respect to $(X_j)_{j=1}^k$ guaranteed
by Lemma \ref{lem:martingale_coupling}. Now, $M'$ is balanced, and so $|M'_j - M'_{j-1}| \le c_j$
for each $j \in [k]$ by Proposition \ref{prop:balanced_implies_bounded}.
Thus, we can apply can apply the (one-sided) Azuma-Hoeffding inequality to ensure
that
\[
\Pr[M'_k \le M_0 - t] \le \exp\left(\frac{-2 t^2}{\sum_{j= 1}^{k} c_j^2}\right),
\]
where we have used that $M'_0 = M_0$.
Returning to $M$, observe that
\[
\Pr[M_k \le M_0 - t] \le \Pr[\text{$M_k \le M_0 -t$ and $X \in \Gamma$}] + \Pr[ X \not \in \Gamma].
\]
Moreover, if $X \in \Gamma$, then $M'_k \le M_k$. Thus,
$\Pr[\text{$M_k \le M_0 -t$ and $X \in \Gamma$}] \le \Pr[M'_k \le M_0 -t]$, 
and so the theorem follows after combining the above equations.
\end{proof}

\section{Proving Theorem \ref{mainsemirandom}}  \label{sec:proof_semi_random}
As an application of Theorem \ref{stealing}, we prove that the properties $\cM$ and $\cH$ admit sharp thresholds in the semi-random graph process. 
In order to prove this, we first establish the existence of sharp thresholds for the \textit{approximate} properties $\cM'$ and $\cH'$, and then we transfer
these thresholds to $\cM$ and $\cH$, respectively. \CM{Note that many of the definitions (respectively, lemmas)
we introduce (respectively, prove) in this section apply to the $\scr{D}$-process.
While our main application is in proving \Cref{mainsemirandom} -- which \textit{is} specific to the semi-random graph process -- we develop our techniques for the $\scr{D}$-process whenever possible.}

Suppose we are given a property $\cP$ in the $\scr{D}$-process, and $m:= m_{\scr{P}}(1/2, n)$. We say that $\cP'$ is an \textbf{approximate property} of $\cP$ if $\cP\subseteq \cP'$,
and for any $G_0\in \cP'$, we can play the $\scr{D}$-process starting with $G_0$ and obtain a graph in $\cP$ in $o(m)$ steps w.h.p.
\begin{lemma}\label{transferapprox}
Let $\cP$ be a property and let $\cP'$ be an approximate property of $\cP$ in the $\scr{D}$-process. If $m^*$ is a sharp threshold for $\cP'$, then $m^*$ is also a sharp threshold for $\cP$ (in the $\scr{D}$-process).
\end{lemma}
	\begin{proof}[Proof of Lemma \ref{transferapprox}]
Recall that $m:= m_{\scr{P}}(1/2,n)$ is the minimum number of steps
 needed to ensure that $\scr{P}$ is satisfied with probability at least $1/2$ in the $\scr{D}$-process.
		For any strategy $\cS$, since $\scr{P} \subseteq \scr{P}'$, 
  we have that $T_{\cS,\cP'}\le T_{\cS, \cP}$. Therefore, for any constant $\eps>0$,
		\begin{equation}\label{transferlower}
		    \bP[ T_{\cS, \cP}\le (1-\eps) \st]\le \bP[ T_{\cS, \cP'}\le (1-\eps) \st]=o(1),
		\end{equation}
and so $m^*$ satisfies the second sharp threshold property for $\scr{P}$.
  
Now, by assumption there exists a strategy $\cS'$ for $\cP'$ such that  $\bP[ T_{\cS', \cP'}\le (1+\eps/2) \st]=1-o(1)$. Consider the strategy $\cS$ for $\cP$ that follows $\cS'$ until we obtain a graph in $\cP'$, and then obtains a graph in $\cP$ in $o(m)$ additional steps w.h.p. (this possible from the definition of approximate property). It follows that

  \begin{equation}\label{transferupper}
		     \bP[ T_{\cS, \cP}\le (1+\eps/2) \st+o(m)]\ge (1-o(1))\bP[ T_{\cS', \cP'}\le (1+\eps/2) \st]     =1-o(1).
		\end{equation}
Observe that \eqref{transferupper} \textit{almost} establishes $m^*$ satisfies the first sharp threshold property for $\scr{P}$, however we must control the $o(m)$ term (note that this term does \textit{not} depend on $\eps$). It suffices to show that $m = (1 + o(1))m^*$. In order to see this, observe that due to the definition of $m =m_{\scr{P}}(1/2,n)$, \eqref{transferlower} and \eqref{transferupper} imply
that for each $\eps > 0$,
\[
(1 - \eps) m^* \le (1 - o(1)) m\le (1 + \eps/2) m^*.
\]
Since this holds for each $\eps > 0$, $m = (1 + o(1)) m^*$ as required, and so the proof is complete.
\end{proof}

		Let $\cM'$ be the property of having a matching that contains at least $n-n^{0.99}$ vertices, and let $\cH'$ be the property of having a path of length at least $n-n^{0.99}$. 
When restricted to the semi-random graph process, the following ``clean-up'' algorithm results of Gao et. al. \cite{gao2021perfect,gao_fully_2022} show that $\cM'$ and $\cH'$ are approximate properties of $\cM$ and $\cH$, respectively. 
	\begin{lemma}[\cite{gao2021perfect}, and [Lemma 2.5, \cite{gao_fully_2022}]\label{cleanupmatching} \label{cleanuppath}
		Suppose $G_0$ is a graph with a matching (respectively, a path) that saturates $n-o(n)$ vertices. If we start the semi-random graph process with $G_0$, then there exists a strategy that constructs $G'\in \cM$ (respectively, $G' \in \cH$) in $o(n)$ steps \whp
	\end{lemma}
\begin{remark}
    We state the quantitative versions of the clean-up algorithms in Appendix \ref{sec:clean_up_algorithms}, as these will be useful in the second part of the proof of Theorem \ref{mainsemirandom}.
\end{remark}
	
It is not difficult to show that in the semi-random graph process, $\cM'$ and $\cH'$ are edge-replaceable. Thus, \Cref{stealing} and \Cref{transferapprox} together imply that $\cM$ and $\cH$ admit sharp thresholds (thus proving the first part of Theorem \ref{mainsemirandom}). To prove the second part of Theorem \ref{mainsemirandom}, it remains to show that there is a sharp threshold of the form $C_{\cP} n$ for both properties. We prove this via an analytic argument in Sections \ref{linearthreshold}, \ref{provingmain}, and \ref{sec:fekete_lemma}. 


In Section \ref{sec:coarse}, we show that non-trivial local properties do \textit{not} admit sharp thresholds in the semi-random graph process. Since Theorem \ref{mainsemirandom} confirms
that two of the most extensively studied global properties admit sharp thresholds, our results suggest that the dichotomy between thresholds for local and global properties that Friedgut \cite{friedgut1999sharp} observed for the Erdős–Rényi random graph also applies to the semi-random graph process.
	
\subsection{Linear Function as a Sharp Threshold}\label{linearthreshold}
\CM{In this subsection, all of our results apply in the full generality of the $\scr{D}$-process.}	
For a given property $\cP$, let $I_n(\cP):=I_n=\min_{\cS_n} \bE \tsn$ where $\cS_n$ is taken over all possible strategies. We focus on properties $\cP$ with $I_n=\Theta(n)$, which we refer to as \textbf{linear} (in $n$). The restriction to the linear regime is a typical feature of results which guarantee the existence of certain limits (see, for example, the interpolation method in \cite{bayati2010combinatorial}). \CM{Recall that \Cref{stealing} ensures that if $\scr{P}$ is edge-replaceable, then $\scr{P}$ has a sharp threshold. By imposing additional analytic conditions on $(\frac{I_n}{n})_{n \ge 1}$,
we can prove the existence of some constant $C>0$ such that $Cn$ is a sharp threshold for $\cP$.}

	\begin{thm} \label{scalinglimit}
		Let $\alpha>0$ and $\delta \in (0,1)$ be constants. Suppose
  that $\cP$ is linear, $n^{\alpha}$-edge-replaceable property satisfying the following conditions for all $n$ sufficiently large: 
		\begin{enumerate}
		    \item \label{item:one_step_expectation} $|I_n-I_{n+1}|<n^\delta$.

            \item \label{item:submax} For all $i \in [n]$ such that $\min(i,n-i)\ge n^{\delta}$,
            \begin{equation*}
		\frac{I_n}{n}\le \max\lrpar{\frac{I_i}{i},\frac{I_{n-i}}{n-i}}+O(n^{\delta-1}). 
		\end{equation*}
		 \end{enumerate}
		Then the limit 
		\begin{equation*}
		\lim_{n\to\infty} \frac{I_n}{n}=:C 
		\end{equation*} 
		exists. Moreover, $Cn$ is a sharp threshold for $\cP$.
	\end{thm}

 In order to prove \Cref{scalinglimit}, we first show that the limit $\lim_{n\to\infty} \frac{I_n}{n}$ exists. We prove a lemma which applies
     to an arbitrary sequence $(a_n)_n$ of reals which is bounded and which satisfies the analogous conditions \ref{item:one_step_expectation}. and \ref{item:submax}. from \Cref{scalinglimit}. 	
	\begin{lemma}\label{lem:fekete_technical}
		Let $\delta\in(0,1)$ be a constant. Suppose that $(a_n)_{n}$ is a bounded sequence of real numbers which satisfies the following conditions for all $n$ sufficiently large:
		\begin{equation}\label{consecutivea}
		|na_n- (n+1)a_{n+1}|\le n^{\delta},
		\end{equation}
        and
        		\begin{equation}\label{submaxa}
		a_n\le \max(a_i,a_{n-i})+O(n^{-\delta})
		\end{equation}
		whenever $\min(i,n-i)\ge n^{1-\delta}$.
		Then $\lim_{n\to\infty} a_n$ exists.
	\end{lemma}
	We remark the similarity between \Cref{lem:fekete_technical} and Fekete's lemma \cite{de1952some}, the latter of which states that if for all $i<n$, \[a_n\le \frac{i}{n} a_i+\frac{n-i}{n} a_{n-i},\] then the limit $\lim_{n\to\infty} a_n$ exists. Indeed the proof of both results are quite similar, and so we defer the proof of \Cref{lem:fekete_technical} to \Cref{sec:fekete_lemma}.

 Before we prove \Cref{scalinglimit}, we require one more technical lemma. Roughly speaking, we shall prove that if $\scr{P}$
 is edge-replaceable, then there exists a strategy which does nearly as well as $I_n$ with polynomially small failure probability.  
	\begin{lemma}\label{concstrat}
		Suppose $\cP$ is a linear $n^{\alpha}$-edge-replaceable property for some fixed $\alpha>0$. Then there exists a strategy $\Sc_n$ such that 
		\[\Pr(T_{\Sc_n}>I_n+n^{1 - \alpha/4})=O(n^{-\alpha/4}).\]
	\end{lemma}
		\begin{remark}
	    Lemma \ref{concstrat} can be proven by a very careful refinement of Theorem \ref{stealing} from Section \ref{sec:main}. In particular, one would have to allow $\theta_1$ and $\theta_2$ to depend on $n$ and approach $0$ and $1$, respectively, sufficiently fast as $n \rightarrow \infty$. We instead opt for a self-contained proof that is simpler. Note that the argument is similar to the proof of \Cref{lem:free_move_coupling_short}.
	\end{remark}


\begin{proof}[Proof of \Cref{concstrat}]
Observe first that the $n^{\alpha}$-edge-replacement procedure guaranteed by the lemma succeeds with probability $1 - o(1/\sqrt{n})$,
and runs for $O(n^{1/2 - \alpha})$ steps. By executing this procedure multiple times independently until the first successful run occurs, we get
an edge-replacement procedure which takes $O(n^{1/2 - \alpha})$ steps in expectation, yet which succeeds with probability $1$. We assume we are working with such a procedure in what follows.  Let us refer to a graph $G'$ on $[n]$ as \textbf{edge-completable}, provided $G' + e \in \scr{P}$ for some edge $e$. 

		Let $\Sc =\Sc_n$ be a strategy on $[n]$ that minimizes expected number of steps needed to satisfy $\cP$ in the $\scr{D}$-process. That is, if $T:= T_{\Sc}$, then $\mb{E}T = I_n$. Observe that for each $t \ge 1$, $\Sc$ selects $Y_t \in X_t$ which minimizes the expected number of steps to satisfy $\scr{P}$, conditional on $X_t$, and the current graph $G_{t-1}$ Thus, we can assume that $\Sc$ is deterministic without loss of generality.


	Let $(X_i)_{i=1}^{\infty}$ be the sequence of random subsets the player receives. For convenience, define $\cH_j :=\sigma(X_1,\dots, X_j)$ for each $j \ge 1$, and let \CM{$\scr{H}_0$ be the trivial sigma-algebra.} Consider the Doob martingale $Z_j:=\bE[T \mid \cH_j]$ for each $j \ge 0$. We shall first prove that
  that $|Z_{j}-Z_{j-1}|= O(n^{1/2-\alpha})$ for each $j \ge 1$. In order to show this, it suffices to argue that for \textit{any} $s_j, s_j'\in \supp(\scr{D})$,
  \begin{equation} \label{eqn:balanced_easy}
      \mb{E}[T \mid (X_i)_{i < j}, X_{j} =s_j] - \mb{E}[ T \mid (X_i)_{i<j}, X_j = s'_j] = O(n^{1/2 - \alpha}).
  \end{equation}
Observe first that since $\Sc$ is deterministic, the event $\{T < j\}$ is $\scr{H}_{j-1}$-measureable. Moreover,
if $T < j$, then $\mb{E}[T \mid (X_i)_{i < j}, X_{j} =s_j] = \mb{E}[ T \mid (X_i)_{i<j}, X_j = s'_j]$, so clearly
\eqref{eqn:balanced_easy} holds when $T < j$.

It remains to prove \eqref{eqn:balanced_easy} when $T \ge j$.
Our approach is to consider a ``stolen" strategy $\scr{S}' = \scr{S}_n'$ for $\scr{P}$ defined on the same sequence $(X_i)_{i=1}^{\infty}$, and which constructs graphs $(G'_i)_{i=0}^{\infty}$. We design the strategy so that its ``winning time'' $T':=T_{\scr{S}'}$ is easily compared to $\mb{E}[ T \mid (X_i)_{i<j}, X_j = s'_j]$, which will allow
us to prove \eqref{eqn:balanced_easy}.


In each step $i \ge 1$, $\scr{S}'$ selects $Y'_i \in X_i$ as follows: If $i \le j-1$, $\scr{S}'$ chooses $Y'_i \in X_i$ in the same way as $\Sc$. If $i = j$, $\scr{S}'$
passes on the round (does not select an edge). Else if $i > j$, then there are two cases to consider: While $G'_{i-1}$
is \textit{not} edge-completable,  $\scr{S}'$ ``pretends'' that $X_j = s'_j$ when copying the decisions of $\Sc$. That is, $Y'_i \in X_i$ is chosen in the same way as
$\Sc$ when given $(X_1, \ldots , s'_j, \ldots X_i)$. Once $G'_{i-1}$ becomes edge-completable, $\scr{S}'$ executes the edge-replacement procedure and returns the resulting graph.

First observe that since $\scr{S}'$ always passes on step $j$, 
\begin{equation} \label{eqn:flip_subset}
\bE[T' \mid (X_i)_{i < j}, X_j = s_j] = \bE[T' \mid (X_i)_{i< j}, X_j=s_j'].
\end{equation}
Moreover, we claim that
\begin{equation} \label{eqn:expected_conditional}
\bE[T' \mid (X_i)_{i< j}, X_j=s_j'] \cdot \bm{1}_{[T \ge j]} = \bE[T \mid (X_i)_{i< j}, X_j=s_j'] \cdot \bm{1}_{[T \ge j]} + O(n^{1/2 -\alpha}).
\end{equation}

In order to prove \eqref{eqn:expected_conditional}, we first observe that if $X_j = s'_j$ and $T \ge j$, then $G'_T$ is edge-completable. 
To see this, note that if $X_j = s'_j$ and $T \ge j$, then we know that $\Sc$ and $\scr{S}'$ made decisions in the
same way up until step $T$, except for when $\scr{S}'$ passed on step $j$. Thus, $G'_T$
has exactly one fewer edge than $G_T$, the graph constructed by $\Sc$ after $T$ steps. Since $G_T \in \scr{P}$, $G'_T$ must be edge-completable. 

On the other hand, if $G'_T$ is edge-completable, then the definition of $\scr{S}'$ implies
that $T' \le T + R$, where $R$ is the number of steps used by the edge-replacement procedure when executed on $G'_T$. Combined with the previous paragraph, we get the following inequality: if $X_j = s'_j$ and $T \ge j$,
then $T' \le T + R$. Since $\{T \ge j\}$ is $\scr{H}_{j-1}$-measurable (here $\scr{H}_{j-1}=\sigma(X_1,\dots, X_{j-1})$), we
can apply conditional expectations to this inequality to conclude that
$$\bE[T' \mid (X_i)_{i< j}, X_j=s_j'] \cdot \bm{1}_{[T \ge j]} \le \bE[T + R\mid (X_i)_{i< j}, X_j=s_j'] \cdot \bm{1}_{[T \ge j]}.$$
But since $T \ge j$, the edge-replacement procedure begins after step $j$,
and so it uses edge subsets which are distinct from $(X_i)_{i \le j}$. Thus, $\mb{E}[R \mid (X_i)_{i< j}, X_j=s_j'] \cdot \bm{1}_{[T \ge j]}= O(n^{1/2-\alpha})$, and so \eqref{eqn:expected_conditional} holds.

To complete the proof of \eqref{eqn:balanced_easy} when $T \ge j$, observe that the optimality of $\Sc$ implies
that $\bE[T \mid (X_i)_{i < j}, X_j = s_j] \le \bE[T'\mid (X_i)_{i < j}, X_j = s_j]$. Thus, by applying \eqref{eqn:flip_subset} followed by \eqref{eqn:expected_conditional}, 
\begin{align*}
\bE[T \mid (X_i)_{i < j}, X_{j} = s_{j}] \cdot \bm{1}_{[T \ge j]} & \le \bE[T' \mid (X_i)_{i < j}, X_{j} = s_{j}] \cdot \bm{1}_{[T \ge j]}\\
                        &= \bE[T' \mid (X_i)_{i< j},X_j=s'_j] \cdot \bm{1}_{[T \ge j]} \\
                        &= \bE[T \mid (X_i)_{i< j}, X_j=s_j'] \cdot \bm{1}_{[T \ge j]} + O(n^{1/2 -\alpha}),
\end{align*}
which, after rearrangement, is precisely \eqref{eqn:balanced_easy} when $T \ge j$.

As we've now verified \eqref{eqn:balanced_easy} (for both cases of $T \le j$ and $T > j$), we know that
$|Z_{j} - Z_{j-1}| = O(n^{1/2 - \alpha})$ for each $j \ge 1$.
We next apply the Azuma-Hoeffding inequality to $(Z_i)_{i=0}^{\infty}$ to get that for any $\gamma>0, \beta>0$,
		\[\Pr(|Z_{\beta n}-\bE T|\ge \gamma \bE T)\le \exp\left(-\Theta\left(\frac{\gamma^2(\bE T)^2}{\beta n(n^{1/2-\alpha})^2}\right)\right)\le \exp\left(-\Theta\left(\frac{\gamma^2n^{2\alpha}}{\beta }\right)\right).\]
	\CM{Observe now that since $T$ is a stopping time with respect to $(\scr{H}_i)_{i \ge 0}$,
 the random variable $T \cdot \bm{1}_{[T \le \beta n]}$ is $\scr{H}_{\beta n}$-measureable. Thus, if $T \le \beta n$, then $Z_{\beta n} = T$.} By applying Markov's inequality to $T$, 
		\[\bP(T\neq Z_{\beta n})\le \bP(T>\beta n)\le  \frac{\bE T}{\beta n}=O(1/\beta).\]
		Combining the two previous equations, 
		\begin{equation*}
		\bP(|T-\bE T|\ge \gamma \bE T)\le  \bP(T\neq Z_{\beta n})+\Pr(|Z_{\beta n}-\bE T|\ge \gamma \bE T)\le O(1/\beta)+\exp\left(-\Theta\left(\frac{\gamma^2n^{2\alpha}}{\beta }\right)\right)
		\end{equation*}
		Now, let $\gamma=n^{-\alpha/4}, \beta=n^{\alpha/4}$. Since $\bE T= I_n$, we get that $\Pr(T>I_n+ n^{1-\alpha/4})=O(n^{-\alpha/4})$,
		which proves the theorem.
	\end{proof}

	\begin{proof}[Proof of Theorem \ref{scalinglimit}]
Let $a_n=I_n/n$ for all $n\ge 1$. Then by \Cref{lem:fekete_technical}, and the assumptions of the theorem, the limit $\lim_{n\to\infty} a_n=C$ exists. 
		
We will now show that $Cn$ is a sharp threshold for $\cP$. Fix an arbitrary constant $\eps>0$. Since $C = \lim_{n\to\infty} I_n/n$,
  we know that for $n$ sufficiently large, $Cn/I_n<(1+\eps)/(1+\eps/2)$. Thus, by \Cref{concstrat}, the strategy $\Sc =\Sc_n$ corresponding to $I_n$ satisfies 
		\[\Pr(T_{\Sc_n}>(1+\eps)Cn)\le  \Pr(T_{\Sc_n}>(1+\eps/2)I_n)=o(1),\]
		and so the first part of the sharp threshold definition is verified.

We now verify the second part of the sharp threshold definition.
		Let $\gamma>0$ be a constant (dependent on $\eps$) chosen sufficiently small such that
  \begin{equation} \label{eqn:gamma_small}
  (1-\eps)\frac{1 + \gamma}{1 - \gamma} < 1.
  \end{equation}
By definition, there exists a strategy which wins in $m(1-\gamma,n)$ steps with probability $1-\gamma$. In the case of failure, we can execute $\Sc$ for an additional $I_n$ steps in expectation. By optimality of $\Sc$, this implies
		that 
		\[ I_n\le (1-\gamma)m(1-\gamma,n)+ \gamma\cdot(m(1-\gamma,n) +  I_n),\]
  and so $I_n \le \frac{1}{1 - \gamma} m(1- \gamma,n)$ after rearrangement. On the other hand, for $n$ sufficiently large we have that $C n/I_n \le (1 + \gamma)$. Therefore, by applying both inequalities,
	\begin{equation} \label{eqn:threshold_inequalities} 
 (1 - \eps) C n\le (1- \eps) (1+ \gamma) I_n \le (1 - \eps) \frac{1 + \gamma}{1 -\gamma} m(1-\gamma,n). \end{equation} On the other hand, since $\scr{P}$ is edge-replaceable, we can apply \Cref{stealing} 
 to argue that $m(1-\gamma,n)$ is a sharp threshold for $\scr{P}$. Therefore, by applying \eqref{eqn:gamma_small} and \eqref{eqn:threshold_inequalities}, we have that for any strategy $\cS_n$,
		\[ \Pr(T_{\cS_n}<(1-\eps)Cn) \le \Pr\left(T_{\cS_n}<(1-\eps) \frac{1 + \gamma}{1 - \gamma} m(1-\gamma,n)\right)=o(1).\]
This establishes the second part of the sharp threshold	definition, and so the proof is complete.
	\end{proof}
	
	\subsection{Proving Theorem \ref{mainsemirandom}} \label{provingmain}
In this section, we restrict our attention to the semi-random graph process and complete
the proof of \Cref{mainsemirandom}. (The definitions we use specific to this process are introduced in \Cref{sec:application_semi_random}). Recall that $\cM'$ and $\cH'$ are the approximate properties
of $\cM$ and $\cH$. Our approach is to verify that $\cM'$ and $\cH'$ satisfy the conditions of
\Cref{scalinglimit}. This will complete the proof of \Cref{mainsemirandom} due to \Cref{transferapprox}.

While verifying condition \ref{item:one_step_expectation}. of \Cref{scalinglimit} for a property $\scr{P}$ is straightforward, verifying condition \ref{item:submax}. is more involved. Because of this, we first define a condition which is more easily verified.
\begin{definition}[Splittable] \label{def:splittable}
Let $\omega = \omega(n)$. We refer to $\scr{P}$ as $\omega$-\textbf{splittable} (or \textbf{splittable} if clear), provided the following guarantee holds: Suppose $i\in [n]$ and $G_0$ is an arbitrary graph on $[n]$.
If the induced graphs $G_0[1,\dots, i]$ and $G_0[i+1,\dots, n]$ each satisfy $\cP$, then there is a strategy with initial graph $G_0$ which satisfies $\cP$ after $\omega$ steps in expectation.
\end{definition}

We now state a technical lemma which allows us to establish condition \ref{item:submax}. from \Cref{scalinglimit}.
	\begin{lemma}\label{lem:fekete_semi}
		Suppose $\cP$ is a linear property which is $n^{\alpha}$-edge-replaceable and $n^{\beta}$-splittable for $\alpha, \beta \in (0,1)$. Then there exists $\delta \in(0,1)$ such that 
		\begin{equation}\label{submax}
		\frac{I_n}{n}\le \max\lrpar{\frac{I_i}{i},\frac{I_{n-i}}{n-i}}+O(n^{\delta-1}) 
		\end{equation}
		for all $i \in [n]$ such that $\min(i,n-i)\ge n^{\delta}$.
	\end{lemma}
	\begin{proof}

In order to prove the lemma, we describe a strategy $\cS$ to be played on $[n]$ whose expected 
number of steps is upper bounded by \eqref{submax}.

First, partition $[n]$ into $A=\{1,\dots, i\}$ and $B=\{i+1,\dots, n\}$. Let
$\Sc_A$ (respectively, $\Sc_B$) be the strategy on vertex set $A$
		(respectively, $B$) guaranteed from Lemma \ref{concstrat} due to the $n^{\alpha}$-edge replacable assumption. Define
		\[N:=\max\lrpar{\frac{n}{i} I_i, \frac{n}{n-i} I_{n-i}}+ n^{1-x}\]
		where  $\delta\in (0,1)$ is a constant to be specified later, $\delta_1 := 1 - \alpha/4$,
        and $x:=\delta(1-\delta_1)/2$.
		During the first $N$ steps, we define $\cS$ to essentially play two games at once: Each time we are given a square in $A$, we choose
		a circle of $A$ via strategy $\Sc_A$, and similarly if we are given a square in $B$, we choose a circle of $B$ via strategy $\Sc_B$.
		
		For $i\ge n^{\delta}$ the number of steps where we play on $A$ is $\bin(N,i/n)$, so
 		\begin{align*}
 		\Pr(\bin(N,i/n)\le I_i+i^{\delta_1}) &\le\exp\lrpar{-\Theta \lrpar{\frac{(n^{-x} i-i^{\delta_1})^2}{Ni/n}} } \\
 		&\le \exp\lrpar{-\Theta\lrpar{\frac{(n^{-x} i)^2}{i}}}= \exp(-\Theta(n^{\delta - 2x}))=O(1/n).
 		\end{align*}
		where the first inequality follows from Chernoff bound, the second inequality follows from $i^{\delta_1}\ll n^{-x} i$, the first equality follows from $i\ge n^{\delta}$ and the last equality follows from $\delta>2x$.
		Therefore, by Lemma \ref{concstrat} the probability that we did not finish the game on $A$ is at most $O(1/n+i^{\delta_1-1})$. Similarly when $n-i\ge n^{\delta}$, the probability that we did not finish the game on $B$ is at most $O(1/n+(n-i)^{\delta_1-1})$. If we finish the game on $A$ and $B$, then since $\scr{P}$ is $n^{\beta}$-splittable, we can construct a graph in $\cP$ on the vertex set $[n]$ in $n^\beta$ expected steps. Otherwise, using the linearity assumption, we just play the game as if the graph is empty and finish in an additional $I_n=O(n)$ expected steps. 
		
		Therefore, for an appropriate choice of $\delta\in (0,1)$ sufficiently close to 1, if $\min(i,n-i)\ge n^{\delta}$ then the total expected number of steps for our strategy is at most 
		\[ N+ n^\delta+ I_n\cdot O\lrpar{\frac{1}{n}+i^{\delta_1-1}+(n-i)^{\delta_1-1}}\le \max\lrpar{\frac{nI_i}{i},\frac{nI_{n-i}}{n-i}}+O(n^{\delta}) ,\]
	     which establishes \eqref{submax}.\end{proof}

We are now ready to prove \Cref{mainsemirandom}. As mentioned, the proof relies on
the quantitative clean-up algorithms of Gao et al. (see \Cref{sec:clean_up_algorithms}).

	\begin{proof}[Proof of Theorem \ref{mainsemirandom}]

We will show that the property $\cM'$ satisfies the condition of Theorem \ref{scalinglimit}. First, it is a linear property since it is known that $I_n/n \in [1/2, 2]$ (as first observed in \cite{ben2020semi}). 
		
  Next, we verify that it is edge-replaceable. For any matching that saturates less than $n-n^{0.99}$ vertices, with probability $1-o(1/\sqrt{n})$, one of the given squares will land on an unsaturated vertex in at most $n^{0.02}$ steps, and we can then form an edge between two unsaturated vertices to form a larger matching. Therefore, $\cM'$ is $n^{0.48}$-edge-replaceable and we can take $\alpha=0.48$.

Condition \ref{item:one_step_expectation}. of \Cref{scalinglimit} is routine to check, but we include the argument here for the sake of completeness. To show $I_{n+1}<I_n+n^\delta$ for any $\delta>0.01$ and large enough $n$, we simply use the strategy that obtains a matching that saturates $n-n^{0.99}$ vertices on the first $n$ vertices in $I_n+O(1)$ expected steps (there are $O(1)$ expected steps where the given square is vertex $n+1$). To obtain a matching that saturates $n+1-(n+1)^{0.99}$ vertices, we simply wait until we are given a square on a vertex that is not saturated by a matching on $n-n^{0.99}$ vertices, which happens in $O(n^{0.01})$ expected steps. Similarly we can show $I_n<I_{n+1}+n^\delta$ by analyzing the optimal strategy that obtains $I_{n+1}$, while ignoring steps that involved vertex $n+1$. 

Next, we check condition \ref{item:submax}. of \Cref{scalinglimit} by verifying $\scr{M}'$ is $n^\beta$-splittable for some $\beta \in (0,1)$. To prove this, we apply \Cref{lem:clean_up_matching_long} of \Cref{sec:clean_up_algorithms} (the quantitative version of the perfect matching clean-up algorithm). Observe that given any $i \in [n]$ from \Cref{def:splittable}, we begin with an initial matching
 of which contains at least $n - 2n^{0.99}$ vertices. By applying \Cref{lem:clean_up_matching_long} with $\eps= 1/n^{0.01}$, we can recover a perfect matching in $O(n^{1-0.005})$ steps in expectation. Thus,
 any $\beta > 1 -0.005$ suffices.   By \Cref{lem:fekete_semi}, we are guaranteed some $\delta \in (0,1)$
 for which condition \ref{item:one_step_expectation}. of \Cref{scalinglimit} is satisfied.

  Therefore $\lim_{n\to\infty} I_n/n=:C_\cM$ exists and $C_\cM n$ is a sharp threshold for the property $\cM'$. Since $\cM'$ is an approximate property of $\cM$, it follows from Lemma \ref{transferapprox} that $C_\cM n$ is a sharp threshold for $\scr{M}$.

		The proof that $\cH'$ satisfies the condition of Theorem \ref{scalinglimit} is similar, and we will only sketch the argument. First, it is a linear property
		since $I_n/n \in [1, 3]$ (as first observed in \cite{ben2020semi}). We will show that if $G_0$ contains a vertex-disjoint union of 2 paths $P_1, P_2$ with total length $\ell-1:=n-n^{0.99}-1$, then we can obtain a path $P$ with length $\ell$ in $O(n^{2/5})$ steps. Without loss of generality, suppose that $P_1$ is the longer path. It is routine to check that in $n^{2/5}$ steps, with probability $1-o(1/\sqrt{n})$, we will receive 2 squares with distance at most $n^{1/4}$ in $P_1$. Therefore, we can consider the strategy that matches all given squares in $P_1$ with one endpoint of $P_2$ if all previous squares are of distance at least $n^{1/4}$ in $P_1$, and then match the first square in $P_1$ that does not satisfy that property with the other endpoint of $P_2$. This strategy will construct a path of length at least $\ell-1-n^{1/4}$ in $O(n^{2/5})$ steps with probability at least $1-o(1/\sqrt{n})$. To extend this to a path of length $\ell$, we simply attached any given unsaturated square with an endpoint of our path. We need to do so $n^{1/4}+1$ times and the expected number of round between receiving unsaturated squares is at most $O(n^{0.01})$. It routinely follows that in, say, $n^{0.27}$ steps we can obtain a path of length $\ell$ with probability at least $1-o(1/\sqrt{n})$. Therefore, $\cH'$ is $n^{0.1}$-edge-replaceable (with room to spare). 

        By an argument similar to the one seen for $\cM'$, it can be seen that condition \ref{item:one_step_expectation}. of Theorem \ref{scalinglimit} holds. Similarly, using \Cref{lem:clean_up_hamiltonian_long} of \Cref{sec:clean_up_algorithms} (quantitative version of Hamiltonian cycle clean-up algorithm), $\scr{H}'$ is splittable, and thus satisfies condition \ref{item:submax}. Therefore, by Theorem \ref{scalinglimit} and Lemma \ref{transferapprox}, there exists a constant $C_\cH>0$ such that $C_\cH n$ is a sharp threshold for $\scr{H}$.
	\end{proof}

\subsection{Proving \Cref{lem:fekete_technical}} \label{sec:fekete_lemma}

	\begin{proof}
	
		Let $\liminf_{n\to\infty} a_n=L$. For convenience let $a_n=a_{\lfloor n\rfloor}$ for non-integer $n$. We note that the inequalities are still true when $i,n$ are not integers (possibly changing $\delta$ and the implicit constants if necessary). Given any $\eps>0$, we will pick some large $N_0$ and $k$ to be specified later. Then for any $N\ge k N_0$, we get that there exists an integer $r$ and $i\in \{0,\dots, k-1\}$ such that 
		\[ (k+i)N_0\le \frac{N}{2^r} \le (k+i+1) N_0\]
		We will then show that
		\begin{align}
		a_{N_0}&\le L+\eps \label{starting}\\
		a_{N}&\le a_{N/2^r}+O(N_0^{-\delta}) \label{divide2}\\
		a_{N/2^r}&\le a_{(k+i+1)N_0}(1+O(1/k))+\ES{O\left(\frac{N_0^{\delta}}{k^{1-\delta}} \right)} \label{multiples}\\
		a_{(k+i+1) N_0}&\le a_{N_0}+O((\log k)N_0^{-\delta}) \label{intdivide}
		\end{align}
		By picking $N_0$ large enough which satisfies \ES{\eqref{consecutivea}, \eqref{submaxa},} \eqref{starting} (which is possible by definition of $L$) and \ES{$k=N_0^{(1+\delta)/(1-\delta)}$}, combining all four inequalities gives us
		\[ a_N\le L+2\eps\]
		for all $N\ge kN_0$.
		
		\ES{To prove \eqref{divide2},}
			we simply use \eqref{submaxa} iteratively, dividing the current index by 2 each time to get that 
			\[ a_{N/2^{\ell}}\le a_{N/2^{\ell+1}} +O((N/2^{\ell})^{-\delta})\]
			for all $\ell=0,\dots, r-1$.
			Hence 
			\[a_{N}\le a_{N/2^r}+ O\lrpar{\sum_{0\le \ell\le r-1} (N/2^{\ell})^{-\delta}}= a_{N/2^r}+O\lrpar{\sum_{i\ge 0} (N_02^i)^{-\delta}}= a_{N/2^r}+O(N_0^{-\delta}).\]

		\ES{To prove \eqref{multiples},}
			we will use \eqref{consecutivea} iteratively.
			We have 
			\[a_{(k+i+1) N_0-1}\le \frac{(k+i+1) N_0}{(k+i+1) N_0-1} a_{(k+i+1) N_0} + ((k+i+1) N_0-1)^{\delta-1}.\]
			\ES{By iterating this $t$ times where $t=(k+i+1) N_0-N/2^r\le N_0$, we have 
            \begin{align*}
                a_{(k+i+1) N_0-t}&\le \frac{(k+i+1) N_0}{(k+i+1) N_0-t} a_{(k+i+1) N_0}+ \frac{\sum_{1\le j\le t} ((k+i+1) N_0-j)^\delta }{(k+i+1) N_0-t}\\ 
                &\le  a_{(k+i+1) N_0}(1+O(1/k))+ O\left(\frac{N_0^{\delta}}{k^{1-\delta}} \right)
            \end{align*}}

		\ES{We will now prove \eqref{intdivide}.}
			For any natural number $j$, we will invoke the following inequalities derived from \eqref{submaxa} : if $j$ is even, then
			\[ a_{j N_0}\le a_{jN_0/2}+O(N_0^{-\delta}),\]
			and if $j$ is odd, then
			\[a_{j N_0}\le \max(a_{(j-1)N_0}, a_{N_0})+O(N_0^{-\delta}).\]
			We will start with $j_0=k+i+1$. Given $j_i$, if $j_i$ is odd let $j_{i+1}=j_i-1$, and if $j_i$ is even let $j_{i+1}=j_i/2$. Clearly $j_{\lceil\log(2k)\rceil}=0$.
			We get that 
			\[ a_{j_\ell N_0}\le \max(a_{j_{\ell+1} N_0}, a_{N_0})+\ell O(N_0^{-\delta})\]
			Therefore  
			\[ a_{j_0 N_0}\le a_{N_0}+ \log(2k)O(N_0^{-\delta})\]
			as desired.
	\end{proof}

	\subsection{Sharp Thresholds Do Not Exist for Local Properties}\label{sec:coarse}
	
	In this section we make some brief observations regarding thresholds for local properties. Given a list of fixed graphs, none of which are forests, we prove that a sharp threshold does \textit{not} exist for the property of containing at least one of these fixed graphs. The results in this section follow from the below result of Behague et al. \cite{behague2021subgraph}, and we include the proofs
	for completeness.
	\begin{thm}[Theorem 1.2 of \cite{behague2021subgraph}]\label{thm:behague}
		Let $H$ be a fixed subgraph of degeneracy $d\ge 2$. Then for any strategy $\cS_n$, 
		if $\tsn$ is the number of rounds needed for $\cS_n$ to build a copy of $H$, then
		\[ \Pr( \tsn\le n^{(d-1)/d}/\omega)=o(1) \]
		for any $\omega\to\infty$.
	\end{thm}
We now prove the following theorem, which implies Theorem~\ref{coarsesimplified}.
	\begin{thm}\label{coarse}
		Let $L$ be a fixed (finite) list of fixed graphs, none of which are forests.  Suppose $d$ is the minimum degeneracy of graphs in $L$. Let $\cP$ be the property of containing a graph in $L$. Then
		\begin{enumerate}
			\item There exists constant $\alpha$ such that for any strategy $\cS$,
			\[ \bP(G^{\cS}_{\alpha n^{(d-1)/d}}\in \cP) \le 1/2 \]
			\item For any constant $\beta$, there exists a constant $\delta=\delta(\beta)>0$ \CM{and a strategy $\scr{S}'$} such that 
			\[ \bP(G^{\cS'}_{\beta n^{(d-1)/d}}\in \cP) \ge \delta \]
		\end{enumerate}
	\end{thm}
	\begin{proof}
		To prove 1., observe that by Theorem \ref{thm:behague} there must exist some $\alpha > 0$ such that for any graph $H\in L$ and any strategy $\cS$, we have that
		\[ \bP(H\in G^{\cS}_{\alpha n^{(d-1)/d}}) \le \frac{1}{2|L|}. \]
		The proof now immediately follows from an application of union bound over all $H\in L$.
		
		Proving 2. is slightly more involved. \CM{Fix a graph $H$ of $L$ of degeneracy $d$.} We first describe the strategy $\scr{S}'$ from \cite{behague2021subgraph}. Since $H$ is $d$-degenerate, we may consider an ordering of the vertices of $H$ (say, $(v_1,\dots, v_k)$) such that $v_i$ has at most $d$ neighbours in $\{v_1,\dots, v_{i-1}\}$. We divide the game into $k$ phases, where in phase $i$ we build the induced graph $H[v_1,\dots, v_i]$. Suppose $v_i$ is adjacent to $v_{i_1},\dots, v_{i_\ell}$ where $\ell\le d$, and we have a copy of $H[v_1,\dots, v_{i-1}]$. To complete phase $i$, if a vertex $v\neq v_1,\dots, v_{i-1}$ is given as a square for the $j^{th}$ time, we match it with $v_{i_j}$. Therefore, if one of such $v$ is given as a square at least $d$ times, then we have successfully built a copy of $H[v_1,\dots, v_i]$. By a standard analysis, this succeeds in $\frac{\beta}{k}n^{(d-1)/d}$ steps with probability $\Omega(1)$. 
		
		Therefore, the probability that $\scr{S}'$ constructs $H$ in $\beta n^{(d-1)/d}$ steps is $\Omega(1)$,
        and so the proof is complete.
	\end{proof}
\ES{
	\begin{proof}[Proof of Theorem~\ref{coarsesimplified}]
	    For the sake of contradiction, suppose that a sharp threshold $m^*=m^*(n)$ exists. Let $\cP$ be the property of containing a graph in $L$. By the definition of sharp threshold, we have that for any strategy $\cS_n$,
        \[ \bP(G^{\cS}_{m^*/2}\in \cP)=o(1). \]
        Thus, taking $\alpha$ and $\beta := \alpha/4$ as in \Cref{coarse}, the second part of
        \Cref{coarse} implies that $m^*/2 \le \alpha n^{(d-1)/d}/4$. From the definition of sharp threshold, we also know that there exists a strategy $\cS'_n$ such that 
        \[ \bP(G^{\cS'_n}_{2m^*}\in \cP)=1-o(1).\]
        But $2m^*\le \alpha n^{(d-1)/d}$, and so by following the strategy of ${\cS'_n}$,
        \[ 1-o(1)=\bP(G^{\cS'_n}_{2m^*}\in \cP)\le\bP(G^{\cS'_n}_{\alpha n^{(d-1)/d}}\in \cP)\le 1/2,\]
        where the final inequality follows from the definition of $\alpha$ in first part of \Cref{coarse}.
        This is a contradiction for sufficiently large $n$, and so the proof is complete.
	\end{proof}}
	
	\begin{remark}\label{forest}
		The case $d=1$ i.e. $H$ is a forest, is trivial. If $H$ is a fixed graph of degeneracy $d=1$, then
		for any strategy $\cS_n$ 
		\[ \Pr(T_{\cS_n}<|V(H)|-1)=o(1),\]
		and there exists a strategy $\cS_n$ such that 
		\[ \Pr(T_{\cS_n}=|V(H)|-1)=1-o(1).\]
	\end{remark}
	
	\section{Conclusion}
	
	Our result allows us to prove the existence of sharp thresholds for edge-replaceable properties in adaptive random graph processes. As we have seen in this paper, being edge-replaceable is a strong enough restriction that natural properties such as $\cM$ and $\cH$ do not satisfy it in the semi-random graph process. We resolved this issue by proving that the properties of interest can be approximated by weaker properties which \textit{are} edge-replaceable. 
	
	It would be of great interest if one can develop more powerful tools to establish sharp thresholds in adaptive random graph processes (or adaptive random processes in general) when edge-replaceable properties do not hold even in the approximate sense. A starting point would
	be to fully resolve the following problem proposed by Ben-Eliezer et al. \cite{ben-eliezer_fast_journal_2020}:
	
	\begin{question} For all $r\ge 2$, does the property of having a $K_r$-factor admit a sharp threshold in the semi-random graph process?
	\end{question}
	
	We resolved this for $r=2$, and it would be interesting to solve this for all constant $r$ or $r = r(n)$ which grows slowly with $n$. More generally, one could ask for a sharp threshold result for the property of containing a certain spanning graph with large minimum degree. This may require new techniques,
	as it seems like such properties are generally not edge-replaceable, even in the approximate sense.
	
	For each property $\cP\in \{\cM,\cH\}$, we have shown that there exists a constant $C_{\cP}$ such that $C_{\cP}n$ is a sharp threshold in the semi-random graph process. That being said, currently only upper and lower bounds are known for $C_{\cP}$ (as implied in \cite{gao2021perfect,gao_fully_2022}).
	\begin{question}
	What is the exact value of $C_{\cP}$ in Theorem \ref{mainsemirandom}?
	\end{question}
	This question currently appears out of reach, as it seems to necessitate designing
	an asymptotically optimal strategy for $\scr{P}$. Our sharp threshold results indicate that
	in order to identify $C_{\cP}$, it suffices to find an optimal strategy 
	which satisfies $\scr{P}$ with (small) constant probability. We hope that this reduction may prove useful in later works.

\subsection*{Acknowledgement.} The authors would like to thank Lutz Warnke for useful discussions on concentration inequalities for martingales.

\bibliographystyle{plain}
\bibliography{refs}

\appendix





\section{Clean-up Algorithms} \label{sec:clean_up_algorithms}

In this section, we state the explicit guarantees of the clean-up algorithms
introduced by Gao et al. \cite{gao_fully_2022,gao2021perfect}. Note that these algorithms
were originally proven to hold w.h.p. in a fixed number of steps (i.e., a Monte Carlo algorithm). By executing independent runs
of such an algorithm until the first success occurs, we can ensure that the clean-up succeeds with probability $1$ while
getting the same asymptotic upper bound on the \textit{expected} number of steps (i.e., we
convert to a Las Vegas algorithm). We make use of these expectation
bounds in the proof of \Cref{mainsemirandom}.


\subsection{Hamiltonian Cycles}

\begin{lemma}[Lemma 2.5, \cite{gao_fully_2022}] \label{lem:clean_up_hamiltonian_long}
Let $0 < \eps = \eps(n) < 1/1000$, and suppose that $P$ is a path on $(1-\eps)n$ vertices
of $[n]$. Then, given $P$ initially, there exists a strategy for the semi-random graph process which builds
a Hamiltonian cycle from $P$ in $O(\sqrt{\eps}n + n^{3/4}\log^2 n)$ steps w.h.p.
Note that the constants hidden in the $O(\cdot)$ notation does not depend on $\eps$.
\end{lemma}

\subsection{Perfect Matchings}

Gao et al. \cite{gao2021perfect} provide a clean-up algorithm with the
following guarantee. For $\eps = 10^{-14}$, if the algorithm is presented a matching $M$ on at least $(1-\eps) n$ vertices of $[n]$, then $M$ can be extended to a perfect matching in at most $100 \sqrt{\eps} n$ steps w.h.p. The analysis of \cite{gao2021perfect} holds when $\eps = \eps(n)$
satisfies $\eps(n) \rightarrow 0$ as $n \rightarrow \infty$, however an additional
term must be added to $100 \sqrt{\eps} n$ if $\eps$ tends to $0$ sufficiently fast.
This is explicitly proven in the first author's thesis \cite{macrury-thesis2023},
and we restate the lemma below for convenience.
\begin{lemma}[Lemma 5.1.8, \cite{macrury-thesis2023}] \label{lem:clean_up_matching_long}
Let $0 < \eps = \eps(n) < 1$, and suppose that $M$ is a matching on $(1-\eps)n$ vertices
of $[n]$. Then, given $M$ initially, there exists a strategy for the semi-random graph process which builds
a perfect matching from $M$ in $O(\sqrt{\eps}n + n^{3/4} \log^2 n)$ steps w.h.p.
Note that the constants hidden in the $O(\cdot)$ notation does not depend on $\eps$.

\end{lemma}

\end{document}